\documentclass[11pt]{amsart}
\usepackage{amssymb}
\usepackage{pgfplots}
\usepackage{csvsimple}
\usepackage{siunitx}
\usepackage{hyperref}
\usepackage{graphicx,color}
\usepackage[normalem]{ulem}
\usepackage{mathrsfs}
\def\R{\mathbb{R}}

\def\Rinf{\R\cup \{+\infty\}}

\def\cC{\mathcal{C}}

\def\cI{\mathcal{I}}

\def\cL{\mathcal{L}}

\def\cN{\mathcal{N}}
\def\cO{\mathcal{O}}
\def\cP{\mathcal{P}}

\def\cS{\mathcal{S}}
\def\cT{\mathcal{T}}

\def\a{\alpha}

\def\g{\gamma}
\def\d{\delta}

\def\l{\lambda}
\def\s{\sigma}
\def\p{\partial}
\def\o{\omega}
\def\veps{\varepsilon}

\def\vphi{\varphi}
\def\O{\Omega}

\def\G{\Gamma}
\def\GD{{\Gamma_D}}
\def\GN{{\Gamma_N}}

\def\DD{{\rm D}}

\def\wto{\rightharpoonup}

\def\tz{\widetilde{z}}

\def\ou{\overline{u}}

\def\tu{\widetilde{u}}

\newcommand{\dv}[1]{\,{\mathrm d}#1}

\newcommand{\wcheck}[1]{#1\hspace{-.8ex}\mbox{\huge {\lower.45ex \hbox{$\textstyle \check{}$}}} \hspace{.5ex}}

\DeclareMathOperator{\diver}{div}

\DeclareMathOperator{\diam}{diam}

\DeclareMathOperator{\trace}{tr}

\let\oldmarginpar\marginpar
\renewcommand\marginpar[1]{
  \oldmarginpar[\raggedleft\footnotesize #1]
  {\raggedright\footnotesize #1}}
\newtheorem{definition}{Definition}
\newtheorem{lemma}[definition]{Lemma}
\newtheorem{proposition}[definition]{Proposition}
\newtheorem{theorem}[definition]{Theorem}

\newtheorem{remark}[definition]{Remark}
\newtheorem{remarks}[definition]{Remarks}
\newtheorem{example}[definition]{Example}

\newtheorem{algorithm}[definition]{Algorithm}

\numberwithin{definition}{section}

\def\RT{{\mathcal{R}T}}
\def\CR{{cr}}

\def\oz{\overline{z}}
\def\ou{\overline{u}}
\def\tg{\widetilde{g}}
\def\tI{\widetilde{I}}

\def\ov{\overline{v}}
\def\nablah{\nabla_{\! h}}

\def\tvphi{\widetilde{\vphi}}

\def\tz{{\widetilde{z}}}

\def\tF{\widetilde{F}}
\def\tS{\widetilde{S}}

\begin{document}
\title[Nonconforming methods for nonlinear PDEs]{Nonconforming
discretizations of convex minimization problems and precise relations to 
mixed methods}
\author{S\"oren Bartels}
\address{Abteilung f\"ur Angewandte Mathematik,  
Albert-Ludwigs-Universit\"at Freiburg, Hermann-Herder-Str.~10, 
79104 Freiburg i.~Br., Germany}
\email{bartels@mathematik.uni-freiburg.de}
\date{\today}
\renewcommand{\subjclassname}{
\textup{2010} Mathematics Subject Classification}
\begin{abstract}
This article discusses nonconforming finite element methods for
convex minimization problems and systematically derives dual mixed 
formulations. Duality relations lead to
simple error estimates that avoid an explicit treatment of nonconformity
errors. A reconstruction formula provides the discrete solution of the 
dual problem via a simple postprocessing procedure which implies a strong
duality relation and is of interest in a~posteriori error estimation.
The framework applies to differentiable and nonsmooth problems, examples
include $p$-Laplace, total-variation regularized, 
and obstacle problems. Numerical experiments illustrate advantages
of nonconforming over standard conforming methods. 
\end{abstract}
 

\maketitle

\section{Introduction}
Mixed finite element methods as introduced in~\cite{RavTho77,BoBrFo13}
provide an attractive framework to  
approximate partial differential equations in divergence form 
since they lead to accurate approximations of fluxes. 
For the Poisson problem it is well understood
that a close connection of mixed methods to nonconforming methods
exists, cf.~\cite{Mari85,ArnBre85}. This is of practical interest since mixed 
finite element methods require the solution of saddle-point problems while 
nonconforming methods lead to positive definite linear systems. 
Moreover, the nonconforming Crouzeix--Raviart element of~\cite{CroRav73}
has proved to be particularly robust and flexible to provide accurate 
approximations for Stokes equations \cite{GirRav86-book}, for nearly incompressible 
Navier--Lam\'e equations \cite{HanLar03}, 
and for singular minimizers related to the Lavrentiev phenomenon
in the calculus of variations \cite{Ortn11}. Another useful feature is 
that the element
is suitable to compute reliable lower bounds for eigenvalue 
problems \cite{ArmDur04,CarGed14}. Further aspects of the Crouzeix--Raviart 
element are addressed in~\cite{Bren15}.
In this article we show that the relation to mixed methods applies to a large class
of convex minimization problems provided an appropriate discretization
is used. From a discrete duality relation we derive quasi-optimal error 
estimates for the
modified discretizations, show that they apply to various nonlinear partial
differential equations and variational inequalities, and illustrate 
the theoretical findings via simulations for certain singular limit settings. 
The results of this article are inspired by recent work on quasi-optimal
convergence rates for nonconforming approximations of total-variation 
regularized problems in~\cite{ChaPoc19-pre}.

\subsection{Convex minimization}
To explain the main ideas we consider a convex variational problem 
defined via a minimization of the energy functional
\[
I(u) =  \int_\O \phi(\nabla u) \dv{x} - \int_\O f u \dv{x},
\]
in a Sobolev space $W^{1,p}_D(\O)$, i.e., subject to homogeneous 
Dirichlet boundary conditions on a
boundary part $\GD\subset \p\O$; we set $\GN=\p\O\setminus \GD$. 
The dual problem is obtained by using the relation $\phi^{**}=\phi$ 
with  the convex conjugate
\[
\phi^*(t) = \sup_{s\in \R^d} s\cdot t - \phi(s)
\]
It consists in maximizing the functional
\[
D(z) = - \int_\O \phi^*(z) \dv{x} 
\]
in the space of vector fields $z\in L^{p'}(\O;\R^d)$ whose 
distributional divergence $\diver z$ belongs to $L^{p'}(\O)$ 
with vanishing normal component on $\GN$ and which
satisfy the constraint
\[
-\diver z = f.
\]
It turns out that solutions are related via 
\[
z = D \phi(\nabla u) \quad \Longleftrightarrow \quad \nabla u = D\phi^*(z),
\]
and satisfy the Euler--Lagrange equation
\[
-\diver D \phi(\nabla u) = f
\]
and the saddle-point system 
\[
D\phi^* (z) - \nabla \l = 0, \quad -\diver z = f,
\]
where $\l$ is the Lagrange multiplier related to the divergence constraint.
One directly verifies that $\l= u$. 

\subsection{Mixed and nonconforming methods}
A low order finite element discretization of the dual problem uses
the Raviart--Thomas finite element space $\RT^0_{\!N}(\cT_h)$ that contains
certain piecewise linear vector fields whose distributional divergence
is given by a piecewise constant function and which have vanishing 
normal component on $\GN$. In the quadratic case with
$\phi(s)=|s|^2/2$ and $\phi^*(t) =|t|^2/2$, corresponding to the Poisson
problem, the numerical method determines a uniquely defined vector field 
$z_h \in \RT^0_{\!N}(\cT_h)$ and an elementwise constant function 
$\ou_h\in \cL^0(\cT_h)$ that solve 
\begin{equation}\label{eq:poisson_mixed_discr}
(z_h,y_h) + (\ou_h,\diver y_h) = 0, \quad (\diver z_h, \ov_h) = -(f,\ov_h)
\end{equation}
for all $(y_h,\ov_h) \in \RT^0_{\!N}(\cT_h)\times \cL^0(\cT_h)$, where 
$(\cdot,\cdot)$ denotes the $L^2$ inner product of functions or 
vector fields with associated norm $\|\cdot\|$. 
The low order nonconforming approximation of the primal problem uses
the Crouzeix--Raviart finite element space $\cS^{1,\CR}_D(\cT_h)$ of
piecewise linear functions that are continuous at midpoints of
sides of elements and vanish at the midpoints of sides belonging to $\GD$. 
It provides a nonconforming approximation of the
Sobolev space $W^{1,2}_D(\O)$. With the piecewise application of the
gradient operator denoted by $\nablah$ we have that the discrete
solution $u_h \in \cS^{1,\CR}_D(\cT_h)$ satisfies 
\[
(\nablah u_h,\nablah v_h) = (f_h,v_h)
\]
for all $v_h\in \cS^{1,\CR}_D(\cT_h)$. It has been shown in~\cite{Mari85}
that the solutions $z_h$ and $u_h$ are related via
\[
z_h|_T(x) = \nablah u_h|_T - \frac{f_h|_T}{d}  (x-x_T)
\]
on every element $T\in \cT_h$ with midpoint $x_T\in T$, 
provided that $f_h=\Pi_{h,0}f $ is the $L^2$ projection of $f$ 
onto $\cL^0(\cT_h)$. Moreover, it follows that 
\[
\ou_h|_T = u_h(x_T) + \frac{f_h|_T}{d^2 |T|} \|x-x_T\|_{L^2(T)}^2. 
\]
Hence, the solution of the mixed finite element method can entirely
be determined by the solution of the nonconforming discretization
and vice versa. We show that
the relations can be generalized and that a modification of the dual
problem simplifies the second equation. 

\subsection{Generalized reconstruction}
We consider the nonconforming discretization of the primal problem
given by the minimization of 
\[
I_h(u_h) =  \int_\O \phi(\nablah u_h) \dv{x} 
- \int_\O f_h u_h \dv{x}
\]
in the set of all $u_h \in \cS^{1,\CR}_D(\cT_h)$. Solutions 
satisfy
\[
\big(D\phi(\nablah u_h),\nablah v_h\big) = (f_h,v_h)
\]
for all $v_h\in \cS^{1,\CR}_D(\cT_h)$. The systematically obtained
discretization 
of the dual problem consists in maximizing the discrete functional
\[
D_h(z_h) =  - \int_\O \phi^*(\Pi_{h,0} z_h) \dv{x}
\]
for $z_h \in \RT^0_{\!N}(\cT_h)$ subject to the constraint
\[
- \diver z_h  = f_h.
\]
The existence of a solution $z_h$ follows from surjectivity properties
of the divergence operator restricted to $\RT^0_{\!N}(\cT_h)$. 
In contrast to consistent discretizations
of the dual problem, here the operator $\Pi_{h,0}$ is included 
in defining $D_h$ leading to discrete duality relations. 
It does not limit the coercivity properties 
of the problem since for divergence-free vector fields in $\RT^0(\cT_h)$
we have that $\Pi_{h,0} y_h = y_h$. In fact, including the operator
$\Pi_{h,0}$ has the interpretation of using quadrature which makes
the numerical realization substantially easier. 
By imposing the divergence constraint via a Lagrange multiplier $\ou_h$ one
finds that optimal pairs $(z_h,\ou_h)\in \RT^0_{\!N}(\cT_h)\times \cL^0(\cT_h)$ 
satisfy the mixed formulation of the dual problem
\[\begin{split}
\big(D\phi^*(\Pi_{h,0} z_h), \Pi_{h,0} y_h\big) + (\ou_h,\diver y_h) &= 0, \\
(\diver z_h,\ov_h) \hspace*{4.1cm}  &= -(f_h,\ov_h), 
\end{split}\]
for all $(y_h,\ov_h) \in \RT^0_{\!N}(\cT_h)\times \cL^0(\cT_h)$. We claim that 
we have
\[
z_h|_T(x) = D\phi(\nablah u_h|_T) - \frac{f_h|_T}{d} (x-x_T)
\]
and 
\[
\ou_h|_T = u_h(x_T)
\]
for all $T\in \cT_h$. To see this, let $\tz_h$ and $\tu_h$ denote the right-hand sides
of the asserted identities for $z_h$ and $\ou_h$. We have that $-\diver \tz_h|_T = f_h|_T$
for all $T\in \cT_h$, and 
\begin{equation}\label{eq:dual_var}
\Pi_{h,0} \tz_h = D\phi(\nablah u_h).
\end{equation}
Hence, for all $v_h\in \cS^{1,\CR}_D(\cT_h)$ we have 
\[
(\tz_h,\nablah v_h) = \big(D\phi(\nablah u_h),\nablah v_h\big)
 = (f_h,v_h)  = (z_h ,\nablah v_h),
\]
where we used an integration-by-parts formula for products of Raviart--Thomas
vector fields and gradients of Crouzeix--Raviart functions. Since 
$\diver (\tz_h-z_h)|_T = 0$ for every $T\in \cT_h$, this identity
implies that $\tz_h-z_h\in \RT^0_{\!N}(\cT_h)$ and in particular that
$\tz_h \in \RT^0_{\!N}(\cT_h)$. Using that $[D\phi^*]^{-1} = D\phi$ we find that
\[
D\phi^*(\Pi_{h,0} \tz_h) = \nablah u_h.
\]
Since $\tu_h$ coincides with the elementwise average of $u_h$ this implies 
that
\[
\big(D\phi^*(\Pi_{h,0} \tz_h),\Pi_{h,0} y_h\big)  + ( \tu_h, \diver y_h) = 0
\]
for all $y_h\in \RT^0_{\!N}(\cT_h)$. Hence, we see that $(\tz_h,\tu_h)$ solves the
mixed finite element formulation and in case of uniqueness coincides with the 
pair $(z_h,\ou_h)$. The crucial identity~\eqref{eq:dual_var} also implies the
important duality relation $I_h(u_h) =  D_h(z_h)$. 
It is also possible to construct the solution $u_h$ of the 
nonconforming discretization from the pair $(z_h,\ou_h)$ solving the mixed
formulation of the dual problem. One directly verifies that this is given by
\[
u_h(x) = \ou_h|_T + D\phi^*(\Pi_{h,0} z_h|_T) \cdot (x-x_T)
\]
for every $T\in \cT_h$ and all $x\in T$. The reconstruction formulas are 
related to discrete Lagrange functionals, e.g., 
\[
L_h(u_h,z_h) 
= \int_\O \nablah u_h \cdot z_h - \phi^*(\Pi_{h,0} z_h) - f_h \Pi_{h,0} u_h \dv{x},
\]
and imply weak and strong discrete duality principles. We note that
related reconstructions in the case of the $p$-Laplace problem have been
identified in~\cite{LiLiCh18}.

\subsection{Error estimates}
The discrete duality relation $I_h(u_h) \ge D_h(z_h)$ provides a 
natural way to derive error estimates. With a coercivity functional 
$\s_{I_h}$ that measures strong convexity properties of $I_h$,
we have for a minimizing $u_h$ that 
\[
\s_{I_h}^2(u_h,v_h)  \le I_h(v_h) - I_h(u_h)
\]
for every $v_h\in \cS^{1,\CR}_D(\cT_h)$. Choosing $v_h = \cI_\CR u$ 
and using $I_h(u_h) \ge D_h(z_h) \ge D_h(\cI_\RT z)$ leads to 
\[
\d_h^2 =  \s_{I_h}^2(u_h,\cI_\CR u ) \le\int_\O \phi(\nablah \cI_\CR u) - f_h \cI_\CR u  
+ \phi^*(\Pi_{h,0} \cI_\RT z) \dv{x}.
\]
Noting that $-\diver \cI_\RT z = f_h$ and using an integration-by-parts
formula show that 
\[
\d_h^2 \le \int_\O \phi(\nablah \cI_\CR u) -
\Pi_{h,0} \cI_\RT z \cdot \nablah \cI_\CR u + \phi^*(\Pi_{h,0}  \cI_\RT z) \dv{x}.
\]
Fenchel's inequality implies that the integrand is 
nonnegative and vanishes if $\nablah \cI_\CR u = D \phi^*(\Pi_{h,0}  \cI_\RT z)$.
The identity $\nablah \cI_\CR u = \Pi_{h,0} \nabla u$ in combination
with Jensen's inequality, the duality relation $I(u) = D(z)$, and an
integration by parts using $-\diver z= f$ lead to
\[\begin{split}
\d_h^2 &\le \int_\O \phi(\nabla u) - \Pi_{h,0} \cI_\RT z \cdot \nabla u 
+ \phi^*(\Pi_{h,0}\cI_\RT z) \dv{x} \\
&= \int_\O -\phi^*(z) + (z- \Pi_{h,0} \cI_\RT z) \cdot \nabla u + \phi^*(\Pi_{h,0}\cI_\RT z) \dv{x}.
\end{split}\]
Finally, using convexity of $\phi^*$, i.e., 
\[
\phi^*(\Pi_{h,0} \cI_\RT z) \le 
\phi^*(z)  - D\phi^*(\Pi_{h,0} \cI_\RT z) \cdot (z-\Pi_{h,0} \cI_\RT z),
\]
and the relation $\nabla u = D\phi^*(z)$ lead to the general error estimate
\[
\d_h^2 \le 
\int_\O \big( D\phi^*(z) - D\phi^*(\Pi_{h,0} \cI_\RT z)\big) 
\cdot (z-\Pi_{h,0} \cI_\RT z)\dv{x}.
\]
In case of a Lipschitz continuous mapping $D\phi^*$ and a regularity
property $z\in W^{1,2}(\O;\R^d)$ we directly deduce a linear convergence
rate for $\d_h$. The estimate and conceptual approach apply however to a 
significantly larger class of 
variational problems including nonsmooth problems. We remark that the same 
upper bound is obtained for the error in approximating
the dual variable, i.e., for $\s_{D_h}^2(z_h,\cI_\RT z)$. 
The error estimate can be improved by incorporating
strong convexity properties of $\phi^*$. For the Poisson problem the derivation 
then corresponds to the estimates 
\[
\|\nablah (u_h - \Pi_{h,0}\cI_\CR u) \|
\le \|\nablah \cI_\CR u - \Pi_{h,0}\cI_\RT z \| 
\le \|z-\Pi_{h,0} \cI_\RT z\|,
\]
i.e., the discretization error related to the nonconforming discretization with
the Crouzeix--Raviart element is controlled by the interpolation error 
for approximating the flux variable in the Raviart--Thomas finite element space.
By making use of interpolation estimates and the triangle inequality this
estimate implies the well known error estimate
\[
\|\nablah u_h - \nabla u \| \le c h \|D^2 u\|.
\]
The derivation given here circumvents the use of a Strang lemma, cf.~\cite{BreSco08-book},
or the decomposition of functions as in~\cite{Gudi10}, to control nonconformity errors.
Another application of duality relations arises in a~posteriori error estimates
for conforming discretizations \cite{Repi00,Brae09}.
If $u_h^c \in W^{1,p}_D(\O)$ is a conforming approximation of the exact solution~$u$ then we
have, assuming for simplicity that $f = f_h$ so that $I_h=I$ and $D_h = D$
on the discrete spaces, that for all $z_h\in \RT^0_{\!N}(\cT_h)$ with
$-\diver z_h = f_h$ we have
\[\begin{split}
\s_I^2(u,u_h^c) 
& \le I(u_h^c) - I(u) \le I(u_h^c) - D(z_h) \\
&= \int_\O \phi(\nabla u_h^c) \dv{x} - z_h \cdot \nabla u_h^c 
+ \phi^*(z_h) \dv{x} =: \frac12 \eta^2(u_h^c,z_h)
\end{split}\]
By Fenchel's inequality the integrand on the right-hand side is nonnegative and 
vanishes if the 
optimality condition $\nabla u_h^c = D\phi^*(z_h)$ holds which can in 
general not be satisfied on the discrete level. The optimal choice
of $z_h$ solves the discrete dual problem which by the arguments given
above is obtained via solving the nonconforming discretization and using the 
reconstructed flux
\[
z_h = D\phi(\nablah u_h) - (f_h/d) (\cdot-x_T).
\]
For the Poisson problem we deduce the estimate
\[
\|\nabla (u_h^c -u) \| \le \eta(u_h^c,z_h) =   \|\nabla u_h^c - z_h \|,
\]
and with the reconstruction relation $z_h = \nablah u_h - (f_h/d) (\cdot-x_T)$ in 
case that $\nabla u_h^c$ is elementwise constant,
\[
\|\nabla (u_h^c -u) \| \le \|\nabla u_h^c - \nablah u_h\|
+  \|(f_h/d)(\cdot -x_T)\| = \widetilde{\eta}(u_h^c,u_h).
\]
The error estimator $\widetilde{\eta}(u_h^c,u_h)$ is 
also efficient, which an application of the triangle inequality and the
equivalence of the conforming and nonconforming method in case of the
Poisson problem show, cf.~\cite{Bren15}.

\subsection{Outline}
The article is organized as follows. We collect various relevant facts 
about Crouzeix--Raviart and Raviart--Thomas finite element spaces in 
Section~\ref{sec:fem_prelim}. In Section~\ref{sec:general} we present 
a general theory leading to an 
error estimate for differentiable convex minimization problems and
a general flux reconstruction formula. Nonsmooth problems including
a quadratic obstacle problem, a total-variation regularized problem, and
an infinity Laplace problem require certain modifications and
are discussed in Section~\ref{sec:nonsmooth}. In preparation of 
numerical experiments we devise iterative algorithms for the
practical realization in Section~\ref{sec:iterative}. The results
of various numerical experiments that reveal certain advantages of
nonconforming methods are presented in Section~\ref{sec:num_ex}.

\section{Finite element spaces}\label{sec:fem_prelim}
Throughout what follows we let $(\cT_h)_{h>0}$ be a sequence of
regular triangulations of 
the bounded polyhedral Lipschitz domain $\O\subset \R^d$ into triangles
or tetrahedra for $d=2$ and $d=3$, respectively. We let $P_k(T)$ denote
the set of polynomials of maximal total degree $k$ on $T\in \cT_h$ and
define the set of discontinuous, elementwise polynomial functions or 
vector fields
\[
\cL^k(\cT_h)^\ell = \{ w_h \in L^\infty(\O;\R^\ell): w_h|_T \in P_k(T) 
\text{ for all }T\in \cT_h\}.
\]
The parameter $h>0$ refers to the maximal mesh-size of the triangulation
$\cT_h$. The set of sides of elements is denoted by $\cS_h$. We let
$x_S$ and $x_T$ denote the midpoints (barycenters) of sides and elements, 
respectively. The $L^2$ projection onto piecewise constant functions or 
vector fields is denoted by
\[
\Pi_{h,0} : L^1(\O;\R^\ell) \to \cL^0(\cT_h)^\ell.
\]
For an elementwise affine function it corresponds to the evaluation 
at element midpoints. Standard notation is used for Sobolev spaces, in 
particular
\[\begin{split}
W^{1,p}_D(\O) &= \{v\in W^{1,p}(\O): v|_\GD = 0 \}, \\
W^q_{\!N}(\diver;\O) &= \{ y \in L^q(\O;\R^d): \diver y \in L^q(\O), \,
y\cdot n = 0 \text{ on }\GN\}.
\end{split}\]
We let $BV(\O)$ denote space of functions in $L^1(\O)$ with finite
total variation denoted $|\DD u|(\O)$. Most estimates derived below follow
from the boundedness of the trace operator 
\[
\trace: W^{1,p}(\O;\R^\ell) \to L^p(\p\O;\R^\ell), \quad v \mapsto v|_{\p\O},
\]
and the Poincar\'e inequality 
\[
\|v-\ov \|_{L^p(\o)} \le c_{p,\o}  \diam(\o) \|\nabla v \|_{L^p(\o)}, \quad 
\ov = |\o|^{-1} \int_\o v \dv{x},
\]
for Lipschitz domains $\o\subset \O$, functions $v\in W^{1,p}(\O;\R^\ell)$ with
mean integral $\ov$ on $\o$, and $1\le p \le \infty$. We occasionally 
make use of indicator functionals, which are for sets $K\subset X$ are defined
by
\[
I_K(s) = 
\begin{cases} 
+\infty & \mbox{for } s\not \in K, \\ 0 & \mbox{for } s\in K,
\end{cases} 
\]
for every $s\in X$. For details on the properties of finite element methods 
listed below we refer the reader to~\cite{Ciar78-book,BoBrFo13,BreSco08-book,ErnGue04-book,Bart16-book}.

\subsection{Crouzeix--Raviart finite elements}
The Crouzeix--Raviart finite element space of lowest order consists 
of piecewise affine functions that are continuous at the midpoints of
sides of elements, i.e., 
\[
\cS^{1,\CR}(\cT_h) = \{v_h \in \cL^1(\cT_h): v_h \text{ continuous in 
$x_S$ for all $S\in \cS_h$} \}.
\]
The space provides nonconforming approximations of Sobolev spaces 
$W^{1,p}(\O)$. The elementwise application of the gradient operator
to a function $v_h\in \cS^{1,\CR}(\cT_h)$ defines an elementwise
constant vector field $\nablah v_h$ via 
\[
\nablah v_h|_T = \nabla (v_h|_T)
\]
for all $T\in \cT_h$. For weakly differentiable functions 
$v\in W^{1,p}(\O)$ we have $\nablah v = \nabla v$. 
The subset of functions vanishing at midpoints
of boundary sides on $\GD$ is denoted by 
\[
\cS^{1,\CR}_D(\cT_h) = \{v_h\in \cS^{1,\CR}(\cT_h): v_h(x_S)=0
\text{ for all $S\in \cS_h$ with $S\subset \GD$}\}.
\]
We note that the jump of a function $v_h\in\cS^{1,\CR}(\cT_h)$ over
an inner element side $S\in \cS_h$ with neighboring elements $T_-,T_+\in \cT_h$,
defined by 
\[
[v_h](x) = v_h|_{T_+}(x) - v_h|_{T_-}(x),
\]
has vanishing integral mean over $S$. Similarly, if 
$v_h \in \cS^{1,\CR}_D(\cT_h)$ then the integral of $v_h|_S$ vanishes
on every boundary side $S\in \cS_h\cap \GD$. A basis of the 
space $\cS^{1,\CR}(\cT_h)$ is given by the functions 
$\vphi_S \in \cS^{1,\CR}(\cT_h)$, $S\in \cS_h$, satisfying 
\[
\vphi_S(x_{S'}) = \d_{S,S'}
\]
for all $S,S'\in \cS_h$. The function $\vphi_S$ vanishes on elements that
do not contain the side $S$ and is continuous with value~1 on $S$. A
quasi-interpolation operator is for $v\in W^{1,p}(\O)$ defined via
\[
\cI_\CR v = \sum_{S\in \cS_h} v_S \vphi_S, \quad v_S = |S|^{-1} \int_S v \dv{s},
\]
Since $\cI_\CR$ is bounded 
and preserves affine functions and averages of gradients, i.e., $\nablah \cI_\CR v = 
\Pi_{h,0} \nabla v$, we have the estimates 
\[
\|v-\cI_\CR v \|_{L^p(\O)} \le  c_\CR h \|\nabla v\|_{L^p(\O)}, 
\quad \|\nablah \cI_\CR v \|_{L^p(\O)} \le \|\nabla v \|_{L^p(\O)}
\]
for all $v\in W^{1,p}(\O)$, $1\le p\le \infty$. Moreover, we have 
$\|\cI_\CR v \|_{L^\infty(\O)} \le c_d \|v\|_{L^\infty(\O)}$  
with $c_d = (d-1)(d+1)$. For $v\in W^{2,p}(\O)$ with $1\le p \le \infty$ we
also have the interpolation estimates
\[
\|v-\cI_\CR v\|_{L^p(\O)} + h \|\nablah \cI_\CR v - \nabla v\|_{L^p(\O)} 
\le c_\CR' h^2 \|D^2 v\|_{L^p(\O)}.
\]
Finally, we note that there exists a  linear
enriching operator 
\[
E_h^\CR: \cS^{1,\CR}_D(\cT_h) \to W^{1,p}_D(\O)
\]
such that
\[
\|\nabla E_h^\CR v_h \|_{L^p(\O)} + h^{-1} 
\|E_h^\CR v_h - v_h\|_{L^p(\O)} \le c_E  \|\nablah v_h\|_{L^p(\O)}
\]
for $1\le p< \infty$, cf.~\cite{Bren15} in case $p=2$ and 
Appendix~\ref{app:enrich} for $p\neq 2$. 

\subsection{Raviart--Thomas finite elements}
The lowest order Raviart--Thomas finite element space is defined as
\[\begin{split}
\RT^0(\cT_h)  = \{y_h\in & W^1(\diver;\O): y_h|_T(x) = a_T + b_T (x-x_T), \\
&  a_T\in \R^d, \, b_T\in \R \text{ for all $T\in\cT_h$} \}.
\end{split}\]
Vector fields in $\RT^0(\cT_h)$ have continuous constant normal components
on element sides. The subset of vector fields with
vanishing normal component on the Neumann boundary $\GN$ is defined as
\[
\RT^0_{\!N}(\cT_h) = \{ y_h\in \RT^0(\cT_h): y_h \cdot n = 0 \text{ on $\GN$}\},
\]
where $n$ denotes the outer unit normal on $\p\O$. A basis of the
space $\RT^0(\cT_h)$ is given by vector fields $\psi_S$, $S\in \cS_h$,
supported on adjacent elements with
\begin{equation}\label{eq:def_rt_basis}
\psi_S(x) = \pm \frac{|S|}{d! |T_\pm|} (z_{S,T_\pm} - x)
\end{equation}
for $x\in T_\pm$ with opposite vertex $z_{S,T_\pm}$ to $S\subset \p T_\pm$.
We have that $\psi_S|_{S'} \cdot n_{S'}=0$
for all sides $S'\neq S$ with unit normal vector $n_{S'}$. If $n_S$ is the 
unit normal vector on $S$ and points from $T_-$ into $T_+$ then we have 
$\psi_S|_S \cdot n_S =1$. A quasi-interpolation operator is for vector fields
$z\in W^{1,1}(\O;\R^d)$ given by 
\[
\cI_\RT z = \sum_{S\in\cS_h} z_S \psi_S, \quad 
z_S = |S|^{-1} \int_S z \cdot n_S \dv{s}.
\]
The operator $\cI_\RT$ is bounded on $C^0(\overline{\O};\R^d)$ and we have 
\[
\|z-\cI_\RT z \|_{L^p(\O)} \le c_\RT h \|\nabla z\|_{L^p(\O)}
\]
and $\diver \cI_\RT z = \Pi_{h,0} \diver z$ for all $z\in W^{1,p}(\O;\R^d)$. 
The latter property implies that the divergence operator defines a surjection
from $\RT^0_{\!N}(\cT_h)$ into $\cL^0(\cT_h)$, provided that constants
are eliminated from $\cL^0(\cT_h)$ if $\GD = \emptyset$. 

\subsection{Orthogonality relations}
An elementwise integration by parts implies that for $v_h\in \cS^{1,\CR}(\cT_h)$
and $y_h\in \RT^0(\cT_h)$ we have the integration-by-parts formula
\begin{equation}\label{eq:int_parts_rt_cr}
\int_\O \nabla_h v_h \cdot y_h \dv{x} + \int_\O v_h \diver y_h \dv{x} 
= \int_{\p\O}  v_h \, y_h \cdot n \dv{s}.
\end{equation}
Here we used that $y_h$ has continuous constant normal components on inner element
sides and that jumps of $v_h$ have vanishing integral mean. If
an elementwise constant vector field $w_h\in \cL^0(\cT_h)^d$ satisfies
\[
\int_\O w_h \cdot \nablah v_h \dv{x} = 0
\]
for all $v_h\in \cS^{1,\CR}_D(\cT_h)$ then its normal components are 
continuous on inner element sides and vanish on the $\GN$, so that 
it belongs to $\RT^0_{\!N}(\cT_h)$.  The following elementary identity
is used repeatedly.

\begin{lemma}[Exchange of projections]\label{la:exchange_projections}
For $z\in W^p_N(\diver;\O) \cap W^{1,1}(\O;\R^d)$ and $u\in W^{1,p}_D(\O)$
and their interpolants $\tz= \cI_\RT z \in \RT^0_{\!N}(\cT_h)$ and
$\tu_h =\cI_\CR u\in \cS^{1,\CR}_D(\cT_h)$ we have
\[
\int_\O \diver z (u - \Pi_{h,0} \tu_h) \dv{x}
+ \int_\O \nabla u \cdot (z - \Pi_{h,0} \tz_h) \dv{x} = 0.
\]
\end{lemma}

\begin{proof}
Since $\diver \tz_h = \Pi_{h,0} \diver z$ and $\nablah \tu_h = \Pi_{h,0} \nabla u$, 
we verify that
\[\begin{split}
 \int_\O \diver z (u - \Pi_{h,0} \tu_h) \dv{x} 
&= - \int_\O z \cdot \nabla u + \diver \tz_h \tu_h \dv{x} \\
&= - \int_\O \nabla u \cdot (z - \Pi_{h,0} \tz_h) \dv{x},
\end{split}\]
which proves the asserted equality.
\end{proof}

\subsection{Convex conjugates}
Given a proper, convex, and lower semicontinuous functional $\phi:\R^d \to \Rinf$ 
the convex conjugate $\phi^*:\R^d \to \Rinf$ is defined via
\[
\phi^*(t) = \sup_{s\in\R^d} t\cdot s - \phi(s).
\]
The function $\phi^*$ is proper, convex, and lower semicontinuous and we have
the relations 
\[
\phi^{**} = \phi, \quad s = D \phi^*\big(D \phi(s)\big),
\]
where the second identy can be generalized to subdifferentials. We refer
the reader to~\cite{Rock70-book} for details and note the Fenchel--Young
inequality which states that for $s,t\in \R^d$ we have 
\[
t\cdot s \le \phi(s) + \phi^*(t)
\]
with equality if and only if $t= D\phi(s)$. 
Certain duality relations can be transferred to discretizations of variational
problems. We provide a modified version and a different proof of an 
important formula identified in~\cite{ChaPoc19-pre}.

\begin{proposition}[Discrete duality]\label{prop:discrete_duality}
Given $\ou_h\in \Pi_{h,0} \cS^{1,\CR}_D(\cT_h) \subset \cL^0(\cT_h)$ we have 
\[\begin{split}
\inf & \Big\{ \int_\O \phi\big(\nablah u_h\big) \dv{x} :  
u_h \in \cS^{1,\CR}_D(\cT_h), \, \Pi_{h,0} u_h = \ou_h \Big\} \\
& \ge
\sup \Big\{ - \int_\O \phi^* \big(\Pi_{h,0} z_h\big) + \ou_h \diver z_h \dv{x}:
z_h \in \RT^0_{\!N}(\cT_h) \Big\}.
\end{split}\]
If $\phi \in C^1(\R^d)$ then equality holds. 
\end{proposition}

\begin{proof}
We let $L(\ou_h)$ and $R(\ou_h)$ denote the terms on the left- and right-hand side of the 
asserted inequality and show that $R(\ou_h)\le L(\ou_h)$. For this, let
$u_h\in \cS^{1,\CR}_D(\cT_h)$ with $\Pi_{h,0} u_h = \ou_h$. Given any 
$z_h \in \RT^0_{\!N}(\cT_h)$ we then have that 
\[
- \int_\O \phi^* \big(\Pi_{h,0} z_h\big) + \ou_h \diver z_h \dv{x}
= - \int_\O \phi^* \big(\Pi_{h,0} z_h\big) - \nablah u_h \cdot \Pi_{h,0} z_h \dv{x}.
\]
Hence, only the midpoint values of $z_h$ matter and the supremum is larger 
if it is taken over elementwise constant vector fields $p_h \in \cL^0(\cT_h)^d$. 
This corresponds to computing elementwise the values 
$\phi(\nablah u_h) = \phi^{**}(\nablah u_h)$. Since $u_h$ is arbitrary
with $\Pi_{h,0} u_h = \ou_h$ we deduce that $R(\ou_h) \le L(\ou_h)$. If
$\phi \in C^1(\R^d)$ then an optimal $u_h\in \cS^{1,\CR}_D(\cT_h)$ for
$L(\ou_h)$ satisfies 
\[
\int_\O D\phi(\nablah u_h) \cdot \nablah v_h \dv{x} 
+ \int_\O \mu_h \Pi_{h,0} v_h \dv{x} = 0,
\]
for all $v_h\in \cS^{1,\CR}_D(\cT_h)$, where 
$\mu_h\in \Pi_{h,0} \cS^{1,\CR}_D(\cT_h)\subset \cL^0(\cT_h)$ is 
a Lagrange multiplier related to the constraint $\Pi_{h,0} u_h = \ou_h$.
For $T\in \cT_h$ and $x\in T$ we define  
\[ 
z_h(x) = D\phi(\nablah u_h|_T) + \frac{\mu_h|_T}{d} (x-x_T)
\]
and note that $\diver z_h|_T = \mu_h|_T$.
We choose an arbitrary element $\tz_h \in \RT^0_{\!N}(\cT_h)$ with $\diver \tz_h = \mu_h$ 
and verify that the elementwise constant vector field $z_h -\tz_h$ satisfies
\[
\int_\O (z_h -\tz_h) \cdot \nablah v_h \dv{x} 
= \int_\O \big(D\phi(\nabla u_h) - \tz_h\big) \cdot \nablah v_h \dv{x}
= 0
\]
for all $v_h\in \cS^{1,\CR}_D(\cT_h)$, i.e., $z_h -\tz_h\in \RT^0_{\!N}(\cT_h)$ and in particular 
$z_h \in \RT^0_{\!N}(\cT_h)$. The identity $\Pi_{h,0} z_h = D\phi(\nablah u_h)$
implies that
\[
\phi^*(\Pi_{h,0} z_h) = \Pi_{h,0} z_h \cdot \nablah u_h - \phi(\nablah u_h).
\]
An integration over $\O$ and the integration-by-parts formula~\eqref{eq:int_parts_rt_cr}
lead to 
\[
\int_\O \phi(\nablah u_h)\dv{x} = - \int_\O \phi^*(\Pi_{h,0} z_h) + \ou_h \diver z_h \dv{x},
\]
which implies that $L(\ou_h) = R(\ou_h)$. 
\end{proof}
 
\begin{remark}
The condition $\phi \in C^1(\R^d)$ can be avoided provided there exists
a sequence of regularizations $\phi_\veps$ of $\phi$ such that $\phi_\veps$
and $\phi_\veps^*$ converge uniformly to $\phi$ and $\phi^*$ on their 
domains. This applies, e.g., to the truncated regularization 
$\phi_\veps(s) = \min\{|s|-\veps/2,|s|^2/(2\veps)\}$ of the modulus 
for which we have $\phi^*_\veps(t) = I_{K_1(0)}(t) + t^2/(2\veps)$, where
$K_1(0)= \{t\in \R^d: |t|\le 1\}$.
\end{remark}

\section{General results}\label{sec:general}
We consider the minimization of the abstract functional
\[
I(u) = \int_\O \phi(\nabla u) \dv{x} + \int_\O \psi(x,u) \dv{x}
\]
in a Sobolev space $W^{1,p}_D(\O)$ for $1< p<\infty$ and $f\in L^{p'}(\O)$. 
We assume that the convex and measurable integrands 
\[
\phi:\R^d \to \R, \quad \psi :\O\times \R \to \Rinf
\]
are such that $I$ is bounded from below, coercive, not 
identical to $+\infty$, and
weakly lower semicontinuous so that the direct method in the calculus of
variations implies the existence of a solution $u\in W^{1,p}_D(\O)$. 
The dual problem consists in maximizing the functional
\[
D(z) = -\int_\O \phi^*(z) \dv{x} - \int_\O \psi^*(x,\diver z) \dv{x}
\]
in the space $W^{p'}_{\!N}(\O;\diver)$ with $p'=p/(p-1)$ and we assume that
a solution exists. We also assume the strong duality relation
\[
\inf_{u\in W^{1,p}_D(\O)} I(u) = \sup_{z\in W^{p'}_{\!N}(\O;\diver)} D(z)
\]
to hold and refer the reader to, e.g.,~\cite{AtBuMi06-book,Rock70-book}, for conditions 
leading to this equality. We recall that in this case we have the relations
\[
z = D\phi(\nabla u), \quad \diver z = D \psi(u)
\]
for solutions $u$ and $z$, where $D\psi$ stands for the derivative of
$\psi$ with respect to the second argument. The derivatives can be replaced
by subdifferentials. 

\subsection{Discrete duality}
The discrete primal problem is defined by minimizing the functional 
\[
I_h(u_h) = \int_\O \phi(\nablah u_h) \dv{x} + \int_\O \psi_h(x,\Pi_{h,0} u_h) \dv{x}
\]
in the nonconforming finite element space $\cS^{1,\CR}_D(\cT_h)$ with
suitable convex approximations $\psi_h$ of $\psi$ that are elementwise constant
with respect to the first argument.
The corresponding discrete dual problem consists in maximizing the
functional
\[
D_h(z_h) = -\int_\O \phi^*(\Pi_{h,0} z_h) \dv{x} 
- \int_\O \psi_h^* (x,\diver z_h) \dv{x}
\]
in the set $\RT^0_{\!N}(\cT_h)$.

\begin{proposition}[Duality relations]\label{prop:discrete_dual_strong}
The discrete primal and dual problems satisfy the
duality relation 
\[
\inf_{u_h\in \cS^{1,\CR}_D(\cT_h)} I_h(u_h) 
\ge \sup_{z_h\in \RT^0_{\!N}(\cT_h)} D_h(z_h).
\]
If $\phi$ and $\psi$ are differentiable 
then solutions $u_h$ and $z_h$ are related via
\[
z_h(x) = D \phi\big(\nabla_h u_h|_T\big) + d^{-1}  D \psi_h \big(x,u_h(x_T)\big) (x-x_T)
\]
for every $T\in\cT_h$ and $x\in T$. The pair 
$(z_h,\ou_h) \in \RT^0_{\!N}(\cT_h)\times \cL^0(\cT_h)$
with $\ou_h|_T = u_h(x_T)$ for all $T\in \cT_h$
solves the corresponding saddle-point problem
\[\begin{split}
\big(D\phi^*(\Pi_{h,0} z_h),\Pi_{h,0} y_h\big) + (\ou_h, \diver y_h) &= 0, \\
(\diver z_h, \ov_h)  - \big(D \psi_h(\ou_h),\ov_h\big)&= 0,
\end{split}\]
for all $(y_h,v_h) \in \RT^0_{\!N}(\cT_h)\times \cL^0(\cT_h)$. Moreover,
in this case strong duality applies, i.e., $I_h(u_h) = D_h(z_h)$. 
\end{proposition}

\begin{proof}
We use the duality formula of Proposition~\ref{prop:discrete_duality} and 
exchange extrema, to verify that, indicating by $u_h,z_h$ arbitrary
functions from the spaces $\cS^{1,\CR}_D(\cT_h)$ and $\RT^0_{\!N}(\cT_h)$,
and abbrevating $\cP_h = \Pi_{h,0} \cS^{1,\CR}_D(\cT_h) \subset \cL^0(\cT_h)$, 
\[\begin{split}
&\inf_{u_h}  I_h(u_h) 
= \inf_{\ou_h\in \cP_h} \inf_{u_h:\, \Pi_{h,0} u_h = \ou_h} \int_\O \phi(\nablah u_h) \dv{x} + \int_\O \psi_h(\Pi_{h,0} u_h) \dv{x} \\
&\quad \ge \inf_{\ou_h \in \cP_h} \sup_{z_h} - \int_\O \phi^* \big(\Pi_{h,0} z_h\big) + \ou_h \diver z_h \dv{x} + \int_\O \psi_h(\ou_h) \dv{x} \\
&\quad\ge \inf_{\ou_h \in \cL^0(\cT_h)} \sup_{z_h} - \int_\O \phi^* \big(\Pi_{h,0} z_h\big) +  \ou_h \diver z_h \dv{x} + \int_\O \psi_h(\ou_h) \dv{x} \\
&\quad\ge \sup_{z_h} \inf_{\ou_h \in \cL^0(\cT_h)} - \int_\O \phi^* \big(\Pi_{h,0} z_h\big) + \ou_h \diver z_h \dv{x} + \int_\O \psi_h(\ou_h) \dv{x}.
\end{split}\]
The infimum is eliminated by using the convex conjugate of $\psi_h$, i.e., 
by noting that  
\[
\psi_h^*(x,t) = \sup_{s\in \R} t \, s - \psi_h(x,s) = - \inf_{s\in \R} -t \, s + \psi_h(x,s), 
\]
we find that 
\[
\inf_{\ou_h\in \cL^0(\cT_h)} - \int_\O \ou_h \diver z_h \dv{x} + \int_\O \psi_h(\ou_h) \dv{x}
= - \int_\O \psi_h^*(\diver z_h) \dv{x}.
\]
This implies that we have 
\[
\inf_{u_h\in \cS^{1,\CR}_D(\cT_h)}  I_h(u_h)  \ge   \sup_{z_h \in \RT^0_{\!N}(\cT_h)}  D_h(z_h).
\]
Assume that $\phi$ and $\psi$ are differentiable, let 
$u_h$ be a solution of the primal problem, and let $z_h$ be defined
as in the proposition. Furthermore, let $\tz_h\in \RT^0_{\!N}(\cT_h)$ be such that
$\diver \tz_h = D \psi_h (\Pi_{h,0} u_h)$.  Since 
$\diver z_h|_T = D \psi_h (u_h(x_T))$ for all $T\in \cT_h$ it follows
that $z_h - \tz_h \in \cL^0(\cT_h)^d$. Using the discrete Euler--Lagrange equations
\[
\int_\O D \phi(\nablah u_h) \cdot \nablah v_h \dv{x} + \int_\O D\psi_h(\Pi_{h,0} u_h) v_h \dv{x} = 0
\]
for all $v_h\in \cS^{1,\CR}_D(\cT_h)$, we find that
\[
\int_\O (z_h-\tz_h) \cdot \nablah v_h \dv{x} 
= \int_\O D \phi(\nablah u_h) \cdot \nablah v_h \dv{x} + \int_\O \diver \tz_h v_h \dv{x} = 0.
\]
Hence $z_h-\tz_h \in \RT^0_{\!N}(\cT_h)$ and in particular $z_h\in \RT^0_{\!N}(\cT_h)$.
Since $\Pi_{h,0} z_h = D\phi(\nablah u_h)$ and $\diver z_h = D\psi_h(\Pi_{h,0}u_h)$ we verify that  
\[\begin{split}
\phi^*(\Pi_{h,0} z_h ) &= \Pi_{h,0} z_h \cdot \nablah u_h - \phi(\nablah u_h), \\
\psi_h^*(\diver z_h) & = \diver z_h  \Pi_{h,0} u_h - \psi_h(\Pi_{h,0} u_h). 
\end{split} \]
Adding these identities and incorporating the integration-by-parts 
formula~\eqref{eq:int_parts_rt_cr} implies that
\[
- \int_\O \phi^*(\Pi_{h,0} z_h) \dv{x} - \int_\O \psi_h^*(\diver z_h) \dv{x}
= \int_\O \phi(\nablah u_h) \dv{x} + \int_\O \psi_h ( \Pi_{h,0} u_h) \dv{x},
\]
i.e., that $I_h(u_h) = D_h(z_h)$. Noting that $D\phi^*(\Pi_{h,0} z_h) = \nablah u_h$
we find that the pair $(z_h,\ou_h)$ solves the saddle-point problem. 
\end{proof}

\begin{remark}
In general, the inclusion $\Pi_{h,0} \cS^{1,\CR}_D(\cT_h) \subset \cL^0(\cT_h)$
is strict, e.g., for $\cT_h = \{T\}$, $\overline{\O} = T$, and $\GD = \p\O$. The implicit 
treatment of Dirichlet boundary conditions in the dual formulation implies that 
strong duality still applies. 
\end{remark}

\subsection{$\G$-convergence of $I_h$}
A general justification of the discrete problems $I_h$ as correct 
discretizations of the functional $I$ is established via a $\G$-convergence
result. For this, we extend the functionals $I$ and $I_h$ to functionals 
$\tI$ and $\tI_h$ on $L^p(\O)$ by formally assigning the value $+\infty$ 
to arguments not belonging to $W^{1,p}_D(\O)$ and $\cS^{1,\CR}_D(\cT_h)$,
respectively. 

\begin{proposition}[$\G$-convergence]
Assume that $\phi:\R^d\to \R$ satisfies 
\[
\big|\phi(s)-\phi(r)\big|\le c_5 \big( 1+ |r| + |s|\big)^{p-1} |r-s|
\]
for all $s,r\in \R^d$ and that $\psi_h:\O\times \R \to \Rinf$ and 
$\psi: \O\times \R \to \Rinf$ are related via 
\[
\psi_h(\cdot,v_h) \to \psi(\cdot,v)
\]
in $L^1(\O)$ as $h\to 0$ whenever $\psi_h(\cdot,v_h) \in L^1(\O)$ for all $h>0$ and
$v_h \to v$ in $L^p(\O)$.
Then, the extended functionals $\tI_h:L^p(\O)\to \Rinf$ are $\G$-convergent
to $\tI:L^p(\O)\to \Rinf$ as $h\to 0$ with respect to strong convergence 
in $L^p(\O)$.
\end{proposition} 

\begin{proof}
Let $(u_h)_{h>0}$ be such that $u_h\in \cS^{1,\CR}_D(\cT_h)$
and $\liminf_{h\to 0} \tI_h(u_h) < \infty$. The assumed coercivity property
implies that for a subsequence we have $\|\nablah u_h\|_{L^p(\O)} \le c$. Incorporating the
enriching operator $E_h^\CR$ shows that there exists $u\in W^{1,p}_D(\O)$ such
that $u_h \to u$ in $L^p(\O)$ and $\nablah u_h \wto \nabla u$ for an
appropriate subsequence as $h\to 0$. 
Convexity of $\phi$ and the assumption on $\psi$ then imply that
\[
\liminf_{h\to 0} I_h(u_h) \le I(u).
\]
Given $u\in W^{1,p}_D(\O)$ we use regularizations of $u$ and the interpolation
operator $\cI_\CR$ to construct a seqence $(u_h)_{h>0}$ with 
$u_h \in \cS^{1,\CR}_D(\cT_h)$ such that 
\[
\|u - u_h\|_{L^p(\O)} + \|\nablah u_h - \nabla u \|_{L^p(\O)} \to 0,
\]
as $h\to 0$. Noting that $\Pi_{h,0} u_h \to u$ in $L^p(\O)$ and using 
the assumed local Lipschitz continuity of $\phi$ 
and the approximability condition on $\psi_h$ and $\psi$ we deduce that
\[\begin{split}
|I_h(u_h) - I(u)| 
& \le c_5 \int_\O (1+|\nabla u| + |\nablah u_h|\big)^{p-1} |\nabla u -\nablah u_h| \dv{x} \\
&\qquad +  \|\psi(u)-\psi_h(\Pi_{h,0} u_h) \|_{L^1(\O)}.
\end{split}\]
This implies that $I_h(u_h) \to I(u)$ as $h\to 0$.
\end{proof}

\subsection{Error estimate}
We next derive an abstract error estimate for the approximation of $I$ 
with the nonconforming discretization $I_h$. We assume that the functionals
$I_h$ provide a uniform strong coercivity property, i.e., with the variational
derivative $\d I_h$, that 
\[
I_h(v_h) + \d I_h (v_h)[w_h-v_h] + \s_{I_h}^2(v_h,w_h) \le I_h(w_h)
\]
for all $v_h,w_h \in \cS^{1,\CR}_D(\cT_h)$ with a definite functional 
$\s_{I_h}$. This implies that minimizers for $I_h$ are unique. 

\begin{theorem}[Discretization error]\label{thm:error_abstract}
For minimizers $u\in W^{1,p}_D(\O)$ and $u_h \in \cS^{1,\CR}_D(\cT_h)$ 
of the functionals $I$ and $I_h$ and a dual solution 
$z\in W^1_N(\diver;\O) \cap W^{1,1}(\O;\R^d)$ we have that
\[\begin{split}
\s_{I_h}^2(u_h,\cI_\CR u)  
&\le \int_\O \big(D\phi^*(z)- D\phi^*(\Pi_{h,0} \cI_\RT z )\big) \cdot (z- \Pi_{h,0} \cI_\RT z) \dv{x} \\
& \quad  + \int_\O \big(D\psi(u)- D\psi_h (\Pi_{h,0} \cI_\CR u)\big) \cdot  (u-\Pi_{h,0} \cI_\CR u) \dv{x} \\
& \quad + \int_\O \psi_h(u)- \psi(u) \dv{x} + \int_\O \psi_h^*(\Pi_{h,0} \diver z) - \psi^*(\diver z) \dv{x}.
\end{split}\]
\end{theorem}

\begin{proof}
The interpolants $\cI_\RT z$ and $\cI_\CR u$ are well defined and we abbreviate 
\[
\tz_h = \cI_\RT z, \quad \tu_h = \cI_\CR u.
\]
By minimality of $u_h$ and the duality relation $I_h(u_h) \ge D_h(\cI_\RT z)$,
we have
\[\begin{split}
\s_{I_h}^2(u_h,\tu_h)  
& \le I_h(\tu_h) -I_h(u_h) \le I_h(\tu_h) - D_h(\tz_h) \\
&= \int_\O \phi(\nablah \tu_h) + \psi_h( \Pi_{h,0} \tu_h) + \phi^*(\Pi_{h,0} \tz_h) + \psi_h^*(\diver \tz_h) \dv{x}. 
\end{split}\]
The identity $\Pi_{h,0} \nabla u = \nablah \tu_h$
in combination with convexity of $\phi$ and Jensen's inequality on
every element leads to 
\[\begin{split}
\s_{I_h}^2(u_h,\tu_h)  
& \le \int_\O \phi(\nabla u) + \psi_h( \Pi_{h,0} \tu_h) 
+ \phi^*(\Pi_{h,0} \tz_h) + \psi_h^*(\diver \tz_h) \dv{x} \\
& = \int_\O \phi(\nabla u) + \psi_h ( \Pi_{h,0} \tu_h) 
+ \phi^*(\Pi_{h,0} \tz_h) + \psi^*(\diver z) \dv{x} \\
& \qquad + \int_\O \psi_h^*(\diver \tz_h) - \psi^*(\diver z) \dv{x}.
\end{split}\]
The duality relation $I(u)=D(z)$ allows us to replace the sum of the 
first and the last term in the first integral and implies that we have
\[\begin{split}
\s_{I_h}^2(u_h,\tu_h)  
& \le \int_\O  - \psi(u) +  \psi_h(\Pi_{h,0} \tu_h) +  \phi^*(\Pi_{h,0} \tz_h) - \phi^*(z) \dv{x}\\
& \qquad + \int_\O \psi_h^*(\diver \tz_h) - \psi^*(\diver z) \dv{x}.
\end{split}\]
We next use convexity of $\phi^*$ and $\psi_h$ at $\Pi_{h,0} \tz_h$ and
$\Pi_{h,0} \tu_h$, respectively, to deduce that
\[\begin{split}
\s_{I_h}^2 & (u_h,\tu_h)  
\le - \int_\O D \psi_h (\Pi_{h,0} \tu_h) \cdot (u-\Pi_{h,0} \tu_h) \dv{x} -\int_\O \psi(u)-\psi_h(u) \dv{x}  \\
& - \int_\O D\phi^*(\Pi_{h,0} \tz_h) \cdot (z- \Pi_{h,0} \tz_h) \dv{x} 
+ \int_\O \psi_h^*(\diver \tz_h) - \psi^*(\diver z) \dv{x}.
\end{split}\]
Using the relations $D\phi^*(z) = \nabla u$ and $\diver z = D\psi(u)$
in Lemma~\ref{la:exchange_projections} implies that 
\[
\int_\O D \phi^*(z) \cdot (z- \Pi_{h,0} \tz_h) \dv{x} 
 = -\int_\O (u- \Pi_{h,0} \tu_h) D\psi(u) \dv{x}.
\]
In combination with the previous estimate we deduce the error bound.
\end{proof}

\begin{remark} 
If $\psi$ is independent of $x\in \O$ and $\psi_h = \psi$ then the last 
two integrals on the right-hand side of the error estimate are nonpositive
by Jensen's inequality. 
\end{remark}

\subsection{Examples}
The abstract error estimates applies to various partial 
differential equations. For linear and quadratic low 
order terms $\psi$ the corresponding error terms simplify, provided the 
approximations $\psi_h$ are suitably chosen. 
 
\begin{proposition}[Low order terms]\label{prop:low_order}
(i) Assume that for $f\in L^q(\O)$ and $f_h = \Pi_{h,0} f$ we have
\[
\psi(x,s) = -f(x) s,  \quad  \psi_h(x,s) = -f_h(x) s.
\]
Then the error estimate of Theorem~\ref{thm:error_abstract} reduces to 
\[
\s_{I_h}^2(\tu_h, u_h)  
\le \int_\O \big(D\phi^*(z)- D\phi^*(\Pi_{h,0} \cI_\RT z )\big) \cdot (z- \Pi_{h,0} \cI_\RT z) \dv{x}.
\]
(ii) Assume that for $g\in L^2(\O)$ and $g_h = \Pi_{h,0} g$ we have
\[
\psi(x,s) = (g(x)-s)^2/2, \quad  \psi_h(x,s) = (g_h(x)-s)^2/2.
\]
Then the error estimate of Theorem~\ref{thm:error_abstract} reduces to 
\[\begin{split}
\s_{I_h}^2(\tu_h, u_h)  
&\le \int_\O \big(D\phi^*(z)- D\phi^*(\Pi_{h,0} \cI_\RT z )\big) \cdot (z- \Pi_{h,0} \cI_\RT z) \dv{x} \\
& \quad +   \|u-\Pi_{h,0} \cI_\CR u\|^2 .
\end{split}\]
\end{proposition}

\begin{proof}
As above we abbreviate $\tu_h = \cI_\CR u$ and $\tz_h = \cI_\RT z$. 
In the first case we have 
\[
\psi^*(x,t) = I_{\{-f(x)\}}(t), \quad \psi_h^*(x,t) = I_{\{-f_h(x)\}}(t).
\]
Hence, the last three integrals in the error estimate of Theorem~\ref{thm:error_abstract}
become 
\[\begin{split}
E_\psi & = -\int_\O (f-f_h) (u-\Pi_{h,0}\tu_h) \dv{x} - \int_\O f_h u - f u \dv{x} \\
& \qquad + \int_\O I_{\{-f_h\}}(\diver \tz_h) - I_{\{-f\}}(\diver z) \dv{x},
\end{split}\] 
so that $E_\psi = 0$ since $f-f_h$ is orthogonal to $\Pi_{h,0} \tu_h$ 
and $\diver z = -f$ and $\diver \tz_h = -f_h$. In the second case we have
\[
\psi^*(x,t) = \frac12 (t+g(x))^2 - \frac12 g(x)^2, \quad 
\psi_h^*(x,t) = \frac12 (t+g_h(x))^2  - \frac12 g_h(x)^2.
\]
The corresponding error terms are given by
\[\begin{split}
E_\psi& = \int_\O \big((u-g)-(\Pi_{h,0} \tu_h -g_h)\big) (u- \Pi_{h,0}\tu_h) \dv{x}
+ \frac12 \int_\O (g_h-u)^2 - (g-u)^2 \dv{x} \\
& \quad 
+ \frac12 \int_\O (\diver \tz_h + g_h)^2 - g_h^2  - (\diver z + g)^2 + g^2\dv{x}.
\end{split}\]
The relation $\diver \tz_h + g_h = \Pi_{h,0}(\diver z + g)$ in combination with
Jensen's inequality and elementary calculations imply that 
\[
E_\psi \le \int_\O (u- \Pi_{h,0} \tu_h)^2 \dv{x}.
\]
This proves the simplified error estimate.
\end{proof}

\begin{remark}
Using the strong convexity of $\psi$ in case~(ii) of 
Proposition~\ref{prop:low_order} the factor~$1$ in front of the term 
$\|u-\Pi_{h,0} \cI_\CR u\|^2$ can be replaced by~1/2. 
\end{remark}

Typical choices for the function $\phi$ correspond to $p$-Laplace equations.

\begin{example}[$p$-Dirichlet problems]\label{ex:p_laplace}
For $1<p<\infty$ let $\phi(s) = |s|^p/p$ and $\psi(x,s) = -f(x)s$ for 
$f\in L^q(\O)$ with $q = p' =p/(p-1)$. Noting that $\phi^*(t) = |t|^q/q$ we define
\[
F(a) = |a|^{(p-2)/2} a, \quad 
\tS(v) = D\phi^*(v) = |v|^{q-2} v, \quad
\tF(v) = |v|^{(q-2)/2} v.
\]
We abbreviate $\tz_h = \cI_\RT z$ and use inequalities from~\cite{DiEbRu07} which are
explained in Appendix~\ref{app:p_laplace} to verify that the error estimate
of Theorem~\ref{thm:error_abstract} becomes
\[\begin{split}
c_p \big\|F(\nablah \cI_\CR u) - F(\nablah u_h)\big\|^2 
& \le \int_\O \big(\tS(z)- \tS(\Pi_{h,0} \tz_h)\big) \cdot (z- \Pi_{h,0} \tz_h)\dv{x} \\
& \le c_p' \|\tF(z) - \tF(\Pi_{h,0} \tz_h)\|^2.
\end{split}\]
The right-hand side can be bounded using techniques from~\cite{DieRuz07} 
provided $\tF(z)\in W^{1,2}(\O;\R^d)$. The results provided there also imply that
$\|F(\nablah \cI_\CR u) - F(\nabla u)\| \le c h \|\nabla F(\nabla u)\|$.  
The estimate confirms error estimates from~\cite{BarLiu93,LiuYan01,DieRuz07}.
\end{example}

\section{Nonsmooth problems}\label{sec:nonsmooth}
We discuss in this section necessary adjustments of the general theory 
to apply it to nondifferentiable problems, where, e.g., well-posedness
and admissibility of modified interpolants has to be ensured.

\subsection{Obstacle problem}
We consider a prototypical obstacle problem defined by minimizing
\[
I(u) = \frac12 \int_\O |\nabla u|^2 \dv{x} - \int_\O f u \dv{x} + I_+(u)
\]
in the set $W^{1,2}_D(\O)$, where $I_+$ is the indicator functional of
functions that are nonnegative almost everywhere. With the
\[
\psi(x,s) = -f(x) u + I_+(s)
\]
we have 
\[
\psi^*(x,t) = I_-(t+f(x)).
\]
The dual problem thus determines a maximizing $z\in L^2(\O;\R^d)$ for
\[
D(z) = - \frac12 \int_\O |z|^2 \dv{x} - I_-(\diver z + f),
\]
where the indicator functional $I_-$ is finite if $\diver z + f$ is nonpositive
as a functional on $W^{1,2}_D(\O)$. We have $z=\nabla u$ and
a complementarity principle implies that $\diver z + f =0$ whenever $u>0$. 
We remark that general obstacles $\chi \in H^1_D(\O)$ can be treated
via a substitution $u = \tu+ \chi$ which leads to a modified function
$f$ provided that $\Delta \chi \in L^2(\O)$. 

\subsubsection*{Discretization}
The discrete primal problem imposes the obstacle constraint at midpoints
of elements, i.e., we consider
\[
I_h(u_h) = \frac12 \int_\O |\nablah u_h|^2 \dv{x} 
- \int_\O f_h u_h \dv{x}
+ I_+(\Pi_{h,0} u_h),
\]
where $f_h = \Pi_{h,0} f$. Proposition~\ref{prop:discrete_dual_strong}
shows that the discrete dual problem consists in determining a 
maximizing vector field $z_h\in \RT^0_N(\cT_h)$ for 
\[
D_h(z_h) = -\frac12 \int_\O |\Pi_{h,0} z_h|^2 \dv{x} - I_- (\diver z_h + f_h).
\]
Adopting the ideas of the general error analysis leads to a quasi-optimal
error estimate. We note that imposing the obstacle condition at midpoints
of elements instead of midpoints of element sides as in the two-dimensional
setting considered in~\cite{CarKoh17} simplifies the error analysis. 

\begin{proposition}[Error estimate]
Let $u\in H^1_D(\O)$ and $u_h\in \cS^{1,\CR}_D(\cT_h)$ are the solutions
of the primal and discrete primal problem, respectively. If
the solution $z\in L^2(\O;\R^d)$ of the dual problem satisfies $z \in H^1(\O;\R^d)$ 
then we have
\[
\|\nablah (u_h - u)\| \le  c h \big(\|D^2 u\| + \|f+\diver z\|\big).
\]
\end{proposition}

\begin{proof}
Throughout this proof we abbreviate $\tu_h = \cI_\CR u$ and $\tz_h= \cI_\RT z$. 
Minimality of $u_h$, strong convexity of $I_h$, and discrete duality imply that 
\[
\d_h^2 =  \frac12 \|\nablah (u_h - \tu_h)\|^2 \le I_h(\tu_h) - D_h(\tz_h).
\]
The relation $\nabla \tu_h = \Pi_{h,0} \nabla u$ in combination with Jensen's 
inequality and the identity $I(u) = D(z)$ show that
\[
\d_h^2 \le -\frac12 \int_\O |z|^2 \dv{x} + \int_\O f u - f_h \Pi_{h,0}\tu_h \dv{x} 
+ \frac12 \int_\O  |\Pi_{h,0} \tz_h|^2 \dv{x}.
\]
Using Lemma~\ref{la:exchange_projections} with $\nabla u = z$ 
and noting $f_h =\Pi_{h,0} f$ leads to 
\[
\d_h^2 \le \int_\O (f+\diver z)  (u - \Pi_{h,0}\tu_h) \dv{x} 
+  \frac12 \int_\O  |z-\Pi_{h,0} \tz_h|^2 \dv{x}.
\]
We abbreviate $\mu = f + \diver z \in L^2(\O)$ and insert $\tu_h = \cI_\CR u$ 
to rewrite the first term on the right-hand side as
\[
\int_\O \mu (u-\Pi_{h,0}\tu_h) \dv{x} = 
\int_\O \mu (u- \tu_h) \dv{x} + \int_\O \mu (\tu_h - \Pi_{h,0}\tu_h) \dv{x}.
\]
To deduce the error estimate it remains to bound the second term  on the
right-hand side. For $T\in \cT_h$ let $\cC_T = \{x\in T: u(x)=0\}$ 
and note that $\l|_{T\setminus \cC_T} = 0$. Since $\nabla u = 0$
almost everywhere on $\cC_T$ and since $\Pi_{h,0} \tu_h|_T = \tu_h(x_T)$  
it follows from $\tu_h(x) = \tu_h(x_T) + \nablah \tu_h|_T \cdot (x-x_T)$
that 
\[\begin{split}
\int_T \mu (\tu_h - \Pi_{h,0}\tu_h) \dv{x} 
&= \int_{\cC_T} \mu \,  \nablah (\tu_h-u) \cdot (x-x_T) \dv{x} \\
&\le h_T \|\mu\|_{L^2(T)} \|\nablah (\tu_h-u)\|_{L^2(T)}.
\end{split}\]  
We thus deduce that 
\[
\d_h^2 \le \|\mu\|\big(\|u - \tu_h\| + h \|\nablah (u-\tu_h)\|\big)
+  \frac12 \|z-\Pi_{h,0} \tz_h\|^2,
\]
which implies the error estimate. 
\end{proof}

\subsubsection*{Flux reconstruction}
The discrete flux $z_h$ can be constructed if a discrete Lagrange
multiplier $\mu_h\in \Pi_{h,0} \cS^{1,\CR}_D(\cT_h)$ is given, i.e., $\mu_h \le 0$ 
is such that
\[
(\mu_h,v_h) = (f_h,v_h) - (\nablah u_h,\nablah v_h)
\]
for all $v_h\in \cS^{1,\CR}_D(\cT_h)$. We then have that
\[
z_h(x) = \nablah u_h|_T - \frac{(f_h - \mu_h)|_T}{d} (x-x_T)
\]
for all $T\in \cT_h$ and $x\in T$.

\subsection{Total variation minimization}\label{sub_sec:tv_min}
Given a function $g\in L^2(\O)$ we consider the primal problem that 
consists in determining a function $u\in BV(\O)\cap L^2(\O)$ which is 
minimial for the functional 
\[
I(u) = |\DD u|(\O) + \frac12 \|u-g\|^2.
\]
The corresponding dual problem determines a maximizing
vector field $z\in W^2_{\!N}(\diver;\O)$ with $\GN=\p\O$ for the functional
\[
D(z) = - \frac12 \|\diver z + g\|^2 + \frac12 \|g\|^2
\]
subject to the pointwise constraint $|z|\le 1$ in~$\O$. From the characterization
\[
|\DD u|(\O) = \sup \Big\{ - \int_\O u \diver z \dv{x} : z\in W^2_{\!N}(\diver;\O), \, 
|z|\le 1 \text{ in $\O$}\Big\},
\]
we obtain the strong duality relation
\[
I(u) = D(z)
\]
for solutions $u$ and $z$ of the primal and dual problems, where $u$ and
$z$ are related via $\diver z = u-g$ and the subdifferential inclusion
$z\in \p |\nabla u|$, cf., e.g.,~\cite{HinKun04}.

\subsubsection*{Discretization}
With $g_h = \Pi_{h,0}g$ the discrete minimization problem is defined as 
the minimization of
\[
I_h(u_h) = \int_\O |\nablah u_h| \dv{x}  + \frac12 \|\Pi_{h,0}u_h-g_h\|^2
\]
in the set of all $u_h \in \cS^{1,\CR}(\cT_h)$. The discrete dual formulation
consists in a maximization of 
\[
D_h(z_h) = - \frac12 \|\diver z_h + g_h\|^2 + \frac12 \|g_h\|^2
\]
in the set $\RT^0_{\!N}(\cT_h)$ subject to the constraints 
$|z_h(x_T)|\le 1$ for all $T\in \cT_h$. Related discretizations have
been used in~\cite{HHSVW19}. The discretization used here is obtained from
Proposition~\ref{prop:discrete_duality} which shows that for every 
$\ou_h \in  \Pi_{h,0} \cS^{1,\CR}_D(\cT_h)$ we have 
\[\begin{split}
&\inf \Big\{\int_\O |\nablah u_h|\dv{x}: \, u_h \in \cS^{1,\CR}(\cT_h), \Pi_{h,0} u_h = \ou_h \Big\}
\\ &\, \ge \sup\Big\{ -\int_\O \ou_h \diver z_h \dv{x}: \,
z_h \in \RT^0_{\!N}(\cT_h), \, |z_h(x_T)| \le 1 \text{ for all $T\in \cT_h$} \Big\},
\end{split}\]
and by using the relation $\diver z_h = \Pi_{h,0} u_h - g_h$  
and arguing as in Proposition~\ref{prop:discrete_dual_strong} 
we obtain the  discrete duality relation
\[
I_h(u_h) \ge D_h(z_h)
\]
for optimal elements $u_h$ and $z_h$, respectively. 
The following quasi-optimal error estimate 
is obtained via constructing appropriate comparison functions. 
It confirms an estimate from~\cite{ChaPoc19-pre} 
in which a discretization using piecewise constant functions and implicitly
incorporating Crouzeix--Raviart elements has been considered. We closely
follow the arguments used therein. 
It is remarkable that the data approximation error $g-g_h$ does not occur explicitly 
which avoids imposing restrictive conditions on $g$.

\begin{proposition}[Error estimate]
Let $u\in BV(\O)\cap L^2(\O)$ and $u_h\in \cS^{1,\CR}(\cT_h)$ be optimal 
for $I$ and $I_h$, respectively. Assume that $g\in L^\infty(\O)$ and there 
exists an optimal $z\in W^2_{\!N}(\diver;\O)$ for
$D$ with $z\in W^{1,\infty}(\O;\R^d)$. We then have that
\[
\|u-u_h\| \le c h^{1/2} \big(\|u\|_{L^\infty(\O)} |\DD u|(\O) + \|g\|\|\nabla z\|_{L^\infty(\O)} \|\diver z\|\big)^{1/2},
\]
where $\|u\|_{L^\infty(\O)} \le \|g \|_{L^\infty(\O)}$.
\end{proposition}

\begin{proof}
The strong convexity properties of $I_h$ and the discrete duality relation
yield that 
\[
\frac12 \|\Pi_{h,0}(u_h - \tu_h) \|^2 \le I_h(\tu_h) - I_h(u_h) \le I_h(\tu_h)- D_h(\tz_h)
\]
for every $\tu_h\in \cS^{1,\CR}(\cT_h)$ and $\tz_h\in \RT^0_{\!N}(\cT_h)$ with 
$|\tz_h(x_T)|\le 1$. Since $g\in L^\infty(\O)$ we have that 
$u\in L^\infty(\O)$ with $\|u\|_{L^\infty(\O)} \le \|g\|_{L^\infty(\O)}$
and Lemma~\ref{la:interpol_primal} below yields the existence
of $\tu_h\in \cS^{1,\CR}(\cT_h)$ with 
\[
I_h(\tu_h) \le I(u) + c h - \frac12 \|g-g_h\|^2
\]
and 
\[
\|u-\tu_h\|_{L^1(\O)} \le c h, \quad \|\tu_h\|_{L^\infty(\O)} \le c. 
\]
Letting $\tz_h\in \RT^0_{\!N}(\cT_h)$ be the function constructed in
Lemma~\ref{la:interpol_dual} below we find that
\[ 
D_h (\tz_h) \ge D(z) - c h - \frac12 \|g-g_h\|^2.
\]
On combining the previous estimates, 
and noting that $I(u) = D(z)$, we deduce that
\[
\frac12 \|\Pi_{h,0} (u_h - \tu_h)\|^2 \le c h.
\]
We incorporate the estimates
\[
\|u-\tu_h\| \le \|u-\tu_h\|_{L^1(\O)}^{1/2} \|u-\tu_h\|_{L^\infty(\O)} \le c h^{1/2}
\]
and
\[
\|v_h- \Pi_{h,0} v_h\|_{L^2(\O)} \le c h \|\nablah v_h\|_{L^1(\O)} \|v\|_{L^\infty(\O)}
\]
for every $v_h\in \cS^{1,\CR}_D(\cT_h)$ to deduce the error bound. 
\end{proof}

\subsubsection*{Modified interpolants}
The following lemma provides the primal comparison function with explicit constants. 

\begin{lemma}[Primal quasi-interpolant]\label{la:interpol_primal}
Given any $u\in BV(\O)\cap L^\infty(\O)$ there exists $\tu_h\in \cS^{1,\CR}(\cT_h)$ 
such that
\[
I_h(\tu_h) \le I(u) + 2 c_d c_\CR h \|u\|_{L^\infty(\O)} |\DD u|(\O) - \frac12 \|g-g_h\|^2.
\]
\end{lemma}

\begin{proof}
We choose a sequence $(u_\veps)_{\veps>0} \in C^\infty(\overline{\O})\cap BV(\O)$ such that 
\[
\|u-u_\veps \|_{L^1(\O)} \to 0, \quad \|\nabla u_\veps \|_{L^1(\O)} \to |Du|(\O), \quad
\|u_\veps \|_{L^\infty(\O)} \to \|u\|_{L^\infty(\O)},
\]
cf.~\cite{AmFuPa00-book,BaNoSa14}.
We then define $\tu_h^\veps = \cI_\CR u_\veps$ and note that
\[
\|\nabla_h \tu_h^\veps\|_{L^1(\O)} \le \| \nabla u_\veps \|_{L^1(\O)} 
\]
We pass to an accumulation point $\tu_h \in \cS^{1,\CR}(\cT_h)$ as $\veps \to 0$
for which we have that
\[\begin{split}
\|\nabla_h \tu_h\|_{L^1(\O)} &\le |\DD u|(\O), \\
\|\tu_h \|_{L^\infty(\O)} &\le  2c_d \|u\|_{L^\infty(\O)}, \\
\|\tu_h  - u\|_{L^1(\O)} &\le c_{\CR} h |\DD u|(\O). 
\end{split}\]
For ease of notation we abbreviate $\ou_h = \Pi_{h,0} \tu_h$ and $g_h = \Pi_{h,0}g$.
We have that 
\[
\|\ou_h-g_h\|^2 = \|\ou_h-g\|^2 - \|g-g_h\|^2
\]
and 
\[
\|\ou_h -g \|^2 = \|u- g\|^2  + \int_\O (\ou_h-u) (\ou_h + u -2g)\dv{x}.
\]
These identities imply that we have
\[\begin{split}
I_h(\tu_h) &= \|\nabla \tu_h\|_{L^1(\O)} + \frac12 \|\ou_h-g_h\|^2 \\
&\le I(u) + \frac12 \|\ou_h-u\|_{L^1(\O)} \|\ou_h+u -2g \|_{L^\infty(\O)} - \frac12 \|g-g_h\|^2.
\end{split}\]
This implies the assertion. 
\end{proof}

A comparison function for the discrete dual problem is constructed in
the following lemma. 

\begin{lemma}[Dual quasi-interpolant]\label{la:interpol_dual}
Given any $z\in W^2_{\!N}(\diver;\O)$ with $z\in W^{1,\infty}(\O;\R^d)$ there
exists $\tz_h \in \RT^0_{\!N}(\diver;\O)$ with $|\tz_h(x_T)|\le 1$ for all
$T\in \cT_h$ and 
\[
D_h(\tz_h) \ge D(z) -c_\RT h L \| g\| \|\diver z \| - \frac12 \|g-g_h\|^2,
\]
with the Lipschitz constant $L$ of $z$. 
\end{lemma}

\begin{proof}
The interpolant $\cI_\RT z$ satisfies $\diver \cI_\RT z = \Pi_{h,0} \diver z$ and
we have with the constant function $\oz|_T = z(x_T)$ that 
\[
|\cI_\RT z(x_T)| \le \|\cI_\RT (z-\oz)\|_{L^\infty(T)} 
+ |\oz| \le c_\RT h L + 1 = \g_h.
\]
Hence, for $\tz_h = \g_h^{-1} \cI_\RT z = \cI_\RT \tz$ with 
$\tz = \g_h^{-1} z$ we have $\tz_h \in \RT^0_{\!N}(\cT_h)$ and
$|\tz_h(x_T)|\le 1$ for all $T\in \cT_h$. Noting that
\[
\diver \tz_h + g_h = \Pi_{0,h} (\diver \tz + g) 
\]
and $\|g\|^2 - \|g_h\|^2 = \|g-g_h\|^2$ we deduce that
\[\begin{split}
D_h(\tz_h) &= - \frac12 \int_\O (\diver \tz_h+ g_h)^2 - g_h^2 \dv{x} \\
&\ge - \frac12 \int_\O (\diver \tz + g)^2 - g^2 \dv{x} - \frac12 \|g-g_h\|^2.
\end{split}\]
Hence, we have that 
\[\begin{split}
D_h(\tz_h) &\ge - \frac12 \int_\O (\diver \tz)^2 + 2 g \diver \tz \dv{x} - \frac12 \|g-g_h\|^2 \\
&= - \frac12 \g_h^{-2} \int_\O (\diver z)^2 \dv{x} - \g_h^{-1} \int_\O g \diver z \dv{x} - \frac12 \|g-g_h\|^2 \\
&\ge - \frac12 \int_\O (\diver z)^2 + 2 g \diver z \dv{x} - (1-\g_h^{-1}) \|g\| \|\diver z \| 
- \frac12 \|g-g_h\|^2,
\end{split}\]
where we also used that $\g_h^{-2} \le 1$. The estimate $1-\g_h^{-1} \le c_\RT h L$
implies the assertion. 
\end{proof}

\begin{remark}
In the absence of the regularity condition $z\in W^{1,\infty}(\O;\R^d)$ 
one can establish $\G$-convergence $I_h \to I$ in $L^1(\O)$. 
Alternatively, one may choose a regularization $z_\veps$ of $z$ 
so that Lemma~\ref{la:interpol_dual} holds with $L_\veps = c \veps^{-1}$. 
An approximability condition on $g$ then implies 
$\|\diver z-\diver z_\veps\| \le \veps$ and leads to the
convergence rate $\cO(h^{1/4})$, cf.~\cite{ChaPoc19-pre}. This rate
has also been obtained in~\cite{Bart12,Bart15-book} for conforming 
approximations and was improved in~\cite{BaNoSa15} in the case
of certain anisotropic functionals. 
\end{remark}

\subsubsection*{Flux reconstruction}
The ideas that lead to the reconstruction of the solution of the
dual problem can be transferred to the nonsmooth situation 
if a regularization of the modulus
function is used to approximate the discrete primal functional $I_h$,
i.e., if $|\cdot|_\veps :\R^d \to \R$ is a differentiable approximation
of euclidean length, then the discrete primal and dual
problems correspond to the Lagrange functional
\[
L_{h,\veps}(u_h,z_h) = -\int_\O u_h \diver z_h + |\Pi_{h,0} z_h|_\veps^* \dv{x}
 + \frac12 \|\Pi_{h,0} u_h - g_h\|^2.
\]
and the relations 
\[
\diver z_h = \Pi_{h,0} u_h - g_h, \quad \nablah u_h \in D|\Pi_{h,0} z_h|_\veps^*,
\]
where the second identity is equivalent to 
\[
\Pi_{h,0} z_h = D|\nablah u_h|_\veps.
\]
If, e.g., $|s| = (|s|^2+\veps^2)^{1/2}$ then we obtain on every $T\in \cT_h$
\[
z_h = \frac{\nablah u_h}{|\nablah u_h|_\veps} + \frac{\Pi_{h,0} u_{h,\veps} -g_h}{d} (\cdot -x_T).
\]

\subsection{Infinity Laplacian}
A variant of the $p$-Laplace problem with $p\to \infty$ arises in problems
of optimal transportation and leads to a minimization of
\[
I(u) = I_{K_1(0)} (\nabla u) -\int_\O f u \dv{x} 
\]
in the space $W^{1,\infty}_D(\O)$ for a given function $f\in L^1(\O)$. 
The dual problem consists in maximizing the functional
\[
D(z) = -\int_\O |z| \dv{x} - I_{\{-f\}} (\diver z) 
\]
in the space $W^1_{\!N}(\diver;\O)$. We refer
the reader to~\cite{Evan99} for existence and strong duality results.

\subsubsection*{Discretization}
We define a discrete approximation of $I$ via 
\[
I_h(u_h) = I_{K_1(0)} (\nablah u_h) - \int_\O f_h u_h \dv{x} 
\]
on the set $\cS^{1,\CR}_D(\cT_h)$ using $f_h = \Pi_{h,0} f$. 
Proposition~\ref{prop:discrete_dual_strong}
implies that the discrete dual problem consists in maximizing the functional
\[
D_h(z_h) = -\int_\O |\Pi_{h,0} z_h| \dv{x} - I_{\{-f_h\}} (\diver z_h) 
\]
in the set of all discrete vector fields $z_h\in \RT^0_{\!N}(\cT_h)$. 
Other discretizations are addressed in~\cite{BarPri07,Ober13,BarPri15,BarSch17,Prye18}.
We have the following approximation result. 

\begin{proposition}[Approximation]\label{prop:approx_inf_laplace}
If a solution $z\in W^1_{\!N}(\diver;\O)$ of the dual problem with $z\in W^{1,1}(\O;\R^d)$
exists and if $u$ and $u_h$ solves the primal and discrete primal problem, respectively, 
then we have 
\[
\big| I_h(u_h) - I(u)| \le c h \big(\|f\|_{L^1(\O)} + \|\nabla z\|_{L^1(\O)} \big).
\]
\end{proposition}

\begin{proof}
Establishing the existence of a discrete solution
$u_h\in \cS^{1,\CR}(\cT_h)$ is 
straightforward by continuity of the discrete 
problem and boundedness of the admissible set. Abbreviating
$\tu_h = \cI_\CR u$ and $\tz_h = \cI_\RT z$ we note that 
$|\nablah \tu_h|\le 1$ and hence 
\[\begin{split}
0 & \le I_h(\tu_h) - I_h(u_h) \le I_h(\tu_h) - D_h(\tz_h) \\
&= -\int_\O  f_h \tu_h \dv{x} + \int_\O |\Pi_{h,0} \tz_h | \dv{x} \\
&= -\int_\O \Pi_{h,0} \tz_h \cdot \nablah \tu_h \dv{x} + \int_\O |\Pi_{h,0} \tz_h | \dv{x} \\
&= -\int_\O \Pi_{h,0} \tz_h \cdot \nabla u \dv{x} + \int_\O |\Pi_{h,0} \tz_h | \dv{x}.
\end{split}\]
The duality relation $I(u)=D(z)$ shows that
\[
- \int_\O |z| \dv{x} = - \int_\O f u\dv{x} = \int_\O z\cdot \nabla u\dv{x}.
\]
This leads to 
\[\begin{split}
0 & \le I_h(\tu_h) - I_h(u_h)  \\
&\le \int_\O (z-\Pi_{h,0} \tz_h) \cdot \nabla u \dv{x} + \int_\O |\Pi_{h,0} \tz_h | -|z| \dv{x} \\
&\le 2 \|z-\Pi_{h,0} \tz\|_{L^1(\O)} \le c h \|\nabla z\|_{L^1(\O)}.
\end{split}\]
Finally, we verify that
\[
I_h(\tu_h) - I(u)  = \int_\O f_h  \Pi_{h,0} \tu_h -f u \dv{x} 
= - \int_\O f (u-\Pi_{h,0} \tu_h) \dv{x},
\]
and deduce the asserted estimate. 
\end{proof}

\begin{remark}\label{rem:inf_laplace_p1}
On right-angled triangulations the conforming $P1$ finite element
method leads to a similar estimate since we have that
\[
\|\nabla \cI_{p1} u \|_{L^\infty(\O)} \le \|\nabla u\|_{L^\infty(\O)}
\]
for every $u\in W^{1,\infty}(\O)$ and hence if 
$u_h^c \in \cS^{1,0}_D(\cT_h)\subset W^{1,\infty}_D(\O)$ is minimal 
for $I_h$ in this set then
\[\begin{split}
0 & \le I(u_h^c) - I(u) = I(u_h^c) 
- I_h(u_h^c)  + I_h(u_h^c) 
- I_h(\cI_{p1} u)  + I_h(\cI_{p1} u) \\ 
&\qquad- I(\cI_{p1} u)  + I(\cI_{p1}u) 
- I(u) \\
&\le \|f-f_h\|_{L^1(\O)} \big(\|u_h^c- \Pi_{h,0} u_h^c\|_{L^\infty(\O)} 
+ \|\cI_{p1} u- \Pi_{h,0}\cI_{p1} u \|_{L^\infty(\O)}\big)  \\
& \qquad + \|f\|_{L^1(\O)} \|u-\cI_{p1} u\|_{L^\infty(\O)},
\end{split}\]
where we used that $f_h = \Pi_{h,0} f$ and $I_h(u_h) \le I_h(\cI_{p1}u )$.
Hence, without additional regularity assumptions we have that 
$|I(u_h^c)-I(u)| \le c h$; if $f\in W^{1,1}(\O)$ and $u\in W^{2,\infty}(\O)$ 
then this can be improved to $\cO(h^2)$. A realistic regularity property
is $u\in W^{4/3,\infty}(\O)$, cf.~\cite{Aron68}. 
\end{remark}

\subsubsection*{Flux reconstruction}
To construct the discrete flux $z_h$ from the solution $u_h$ of the 
nonconforming method for the primal problem we consider
a regularization $|\cdot|_\veps$ of the euclidean length which defines 
regularizations $|\cdot|_\veps^*$ of $I_{K_1(0)}$. We then find that
on every $T\in \cT_h$ we have 
\[
z_h = D |\nablah u_h|_\veps^* - (f_h/d)(\cdot -x_T),
\]
where $z_h$ and $u_h$ are the solutions of the regularized problems. 

\section{Iterative solution}\label{sec:iterative}
To solve the discrete problems we devise iterative algorithms for problems with
sub- and superquadratic growth properties that result from semi-implicit discretizations
of appropriate gradient flows for the primal and dual problem, respectively. 
A gradient flow for the primal minimization problem determines a family
$(u(t))_{t\ge 0} \subset W^{1,p}_D(\O)$ of functions for an initial 
$u^0 \in W^{1,p}_D(\O)$ via $u(0) = u^0$ and 
\[
(\p_t u,v)_* = - \int_\O D\phi(\nabla u) \cdot \nabla v \dv{x} - \int_\O D\psi(u) v\dv{x}
\]
for all $v\in W^{1,p}_D(\O)$ and all $t>0$. To avoid solving nonlinear systems
of equations a semi-implicit discretization in time is used. We consider
the case that $\phi$ only depends on the length of its argument, i.e., 
$\phi(s) = \vphi(|s|)$
with a convex function $\vphi \in C^1(\R_{\ge 0})$. In this case we have
\[
D\phi(s) = \frac{\vphi'(|s|)s}{|s|},
\]
which naturally leads to a semi-implicit treatment. To discretize
the time derivative we use the the backward difference quotient operator
\[
d_t u^k = \tau^{-1} (u^k-u^{k-1})
\]
for a sequence $(u^k)_{k=0,1,\dots}$ and a step-size $\tau>0$.

\begin{algorithm}[Subquadratic case, primal iteration]\label{alg:primal}
Let $u^0\in W^{1,p}_D(\O)$ and choose $\tau,\veps_{\rm stop} >0$, set $k=0$. \\
(1) Compute $u^k\in W^{1,p}_D(\O)$ such that 
\[
(d_t u^k,v)_* + 
\int_\O \frac{\vphi'(|\nabla u^{k-1}|)}{|u^{k-1}|} \nabla u^k \cdot \nabla v \dv{x} 
+ \int_\O D\psi(u) v \dv{x} 
= 0
\]
for all $v\in W^{1,p}_D(\O)$. \\
(2) Stop if $\|d_t u^k\|_* \le \veps_{\rm stop}$; otherwise increase
$k\to k+1$ and continue with~(1).
\end{algorithm}

It is shown below that the iteration is unconditionally energy decreasing and 
convergent if $\vphi$ has subquadratic growth. If this is not the
case then we expect the dual problem to have this property and consider
a gradient descent for $-D$, i.e., we determine a family $(z(t))_{t\ge 0}$
satisfying $z(0)=z^0$ and the constrained evolution equation
\[
(\p_t z,y)_\dagger =  - \int_\O D\phi^*(z) \cdot y \dv{x} 
-  \int_\O D\psi^*(\diver z) \diver y \dv{x}
\]
for all $y\in W^{p'}_N(\diver;\O)$. In case of a linear functional $\psi$,
the differential $D\psi^*$ becomes a subdifferential
and the equation a variational inequality or constrained equation. 
Similarly to the gradient flow for the primal problem we assume
that the integrand is isotropic, i.e.,
$\phi^*(r) = \vphi(|r|)$ with a convex function $\vphi \in C^1(\R_{\ge 0})$. 
In this case we have
\[
D\phi^*(r) = \frac{\vphi'(|r|)r}{|r|}
\]
and the semi-implicit iteration is similar to that of 
Algorithm~\ref{alg:primal}.

\begin{algorithm}[Superquadratic case, dual iteration]\label{alg:dual}
Let $z^0\in W^{p'}_N(\diver;\O)$ and choose $\tau,\veps_{\rm stop} >0$, set $k=0$. \\
(1) Compute $z^k\in W^{p'}_N(\diver;\O)$ such that 
\[
(d_t z^k,y)_\dagger + 
\int_\O \frac{\vphi'(|z^{k-1}|)}{|z^{k-1}|} z^k \cdot  y \dv{x} 
+ \int_\O D\psi^*(\diver z^k) \diver y \dv{x} = 0,
\]
for all $y\in W^{p'}_N(\diver;\O)$. \\
(2) Stop if $\|d_t z^k\|_\dagger \le \veps_{\rm stop}$; otherwise increase
$k\to k+1$ and continue with~(1).
\end{algorithm}

If $\psi(x,s) = -f(x)s$ then the system in Step~(1) includes the constraints
$-\diver z^k = f$ and $\diver y =0$ instead of the integral involving $D\psi^*$.
The algorithms converge for subquadratic growth of $\phi$ and $\phi^*$, respectively.
We adopt arguments from~\cite{BaDiNo18}. 

\begin{proposition}[Unconditional convergence]\label{prop:uncond_conv}
Assume that $r\mapsto \vphi'(r)/r$ is positive, non-increasing, 
and continuous on $\R_{\ge 0}$. If $\phi(s) = \vphi(|s|)$ for all $s\in \R^d$
then the iteration of Algorithm~\ref{alg:primal}
is well-posed, convergent, and monotone with
\[
I(u^\ell) + \tau \sum_{k=1}^\ell \|d_t u^k\|_*^2 \le I(u^0).
\]
If $\phi^*(t) = \vphi(|t|)$ then the iteration of Algorithm~\ref{alg:dual}
is well-posed, convergent, and monotone with
\[
-D(z^\ell) + \tau \sum_{k=1}^\ell \|d_t z^k\|_\dagger^2 \le -D(z^0).
\]
\end{proposition}

\begin{proof}
(i) The conditions on $\vphi$ imply that the iteration is well posed and
that we have
\begin{equation}\label{eq:mon_phi}
\frac{\vphi'(|a|)}{|a|} b\cdot (b-a) \ge \vphi(|b|) - \vphi(|a|) 
+ \frac12 \frac{\vphi'(|a|)}{|a|} |b-a|^2
\end{equation}
for all $a,b\in \R^d$, cf.~Appendix~\ref{app:monotone} for a proof of~\eqref{eq:mon_phi}.
Hence, by choosing $v= d_t u^k$ in Algorithm~\ref{alg:primal} we find
that 
\[
\|d_t u^k\|_*^2 
+ \int_\O \frac{\vphi'(|\nabla u^{k-1}|)}{|u^{k-1}|} \nabla u^k \cdot \nabla d_t u^k \dv{x} 
+\int_\O D\psi(u^k) d_t u^k \dv{x}  = 0.
\]
Using $a=\nabla u^{k-1}$ and $b= \nabla u^k$ in~\eqref{eq:mon_phi}
shows that
\[
\frac{\vphi'(|\nabla u^{k-1}|)}{|\nabla u^{k-1}|} \nabla u^k\cdot \nabla (u^k-u^{k-1})
\ge \vphi(|\nabla u^k|) - \vphi(|\nabla u^{k-1}|).
\]
By combining the last two equations, using convexity of $\psi$, 
and summing over $k=1,2,\dots,\ell$ we deduce the asserted estimate. \\
(ii) If the conditions on $\phi^*$ are satisfied then the arguments used to show~(i)
apply to Algorithm~\ref{alg:dual} and we deduce the estimate. 
\end{proof}

\begin{example}
The conditions of the proposition apply to typical regularized $p$-Dirichlet energies 
$\phi(s) = |s|_\veps^p$ for $\veps>0$, cf.~\cite{BaDiNo18}. Algorithm~\ref{alg:primal}
converges if $1\le p \le 2$ while Algorithm~\ref{alg:dual} converges if $2\le p <\infty$.
\end{example}

\begin{remark}
Note that owing to the semi-implicit discretization the functions $d_t u^k$ 
and $d_t z^k$ are not residuals. If, e.g., $\tu=u^k$ for some $k\ge 0$ 
and the residual $r$ is defined via 
\[
(D \phi_\veps(\nabla \tu),\nabla v) + (D\psi(\tu),v) = (r,v)_*
\]
for all $v\in W^{1,p}_D(\O)$, then by convexity of $I_\veps$ we have
\[
I_\veps(\tu) + (r,v-\tu)_* + \s_I^2(\tu,v) \le I_\veps(v)
\]
for all $v\in W^{1,p}_D(\O)$, where we assume that coercivity holds uniformly 
with respect to $\veps\ge 0$. In case of the $L^2$ scalar product
$(\cdot,\cdot)_*=(\cdot,\cdot)$, and if, e.g., $\s_I^2(\tu,v) \ge (\a_I/2) \|\tu-v\|^2$,
we deduce that
\[
I_\veps(\tu) + \frac{\a_I}{4} \|v-\tu\|^2 \le I_\veps(v) + \frac{1}{\a_I} \|r\|^2.
\]
With the minimizing $u_\veps$ for $I_\veps$ we deduce 
$\|u_\veps-\tu \| \le (2/\a_I) \|r\|$. 
\end{remark}

Two alternative approaches to the iterative solution of the discrete problems
are described in the following remarks.

\begin{remarks}\label{rem:admm_primal_dual}
(i) The ADMM iteration (alternating direction of multiplier method) as 
in~\cite{ForGlo83-book} decouples the gradient operator
from $\phi$ by introducing 
$q=\nabla u$ via a Lagrange multiplier $\l$. With the augmented Lagrange functional
\[
L_\tau(u,q,\l) = \int_\O \phi(q) \dv{x} + \int_\O \psi(u) \dv{x} + (\l,\nabla u - q)_H 
+ \frac{\tau}{2} \| \nabla u -q \|_H^2,
\]
with a suitable Hilbert space norm $H$ and a stabilization parameter 
$\tau>0$, the algorithm successively minimizes $L_\tau$ with respect to $u$ and $q$,
and then performs an ascent step with respect to $\l$. \\
(ii) Primal-dual methods as investigated in~\cite{ChaPoc11}
alternatingly update the variable $u$ and $z$ in the Lagrange functional
\[
L(u,z) = \int_\O - u \diver z - \phi^*(z) + \psi(u) \dv{x}.
\]
via discretizations of $\p_t z = \d_z L(u,z)$ and $\p_t u = -\d_u L(u,z)$ using an
extra\-polated quantity to decouple the equations. The application to Raviart--Thomas
methods is not straightforward due to their nonlocal character.
\end{remarks}

\section{Numerical experiments}\label{sec:num_ex}
In this section we verify the theoretical findings via numerical experiments
and illustrate advantages of nonconforming and mixed methods over standard
conforming methods.

\subsection{Total variation minimization}
We consider the numerical approximation of the functional
\[
I(u) = |\DD u|(\O) + \frac{\a}{2} \|u-g\|^2.
\]
To compare approximations to an exact solution we impose Dirichlet boundary
conditions on $\GD=\p\O$. Although it is difficult to establish a general
existence theory, the error estimates of Section~\ref{sub_sec:tv_min}
carry over verbatimly with $\GN=\emptyset$ provided a minimizer exists. This
is the case in the setting of the following example.

\begin{example}\label{ex:tv_exact}
For $\O\subset \R^d$ and $r>0$ with $B_r(0)\subset \O$, and $g=\chi_{B_r(0)}$ 
the unique minimizer for $I$ subject to Dirichlet boundary conditions is given 
by
\[
u = \max\{0,1-d/(\a r)\} \chi_{B_r(0)}.
\]
If $d\le \a r$ then the Lipschitz continuous vector field 
\[
z(x) = 
\begin{cases}
-r^{-1} x & \text{for } |x|\le r, \\
-rx/|x|^2 & \text{for } |x|\ge r,
\end{cases}
\]
solves the dual problem, cf., e.g.,~\cite{Bart15-book}.
We use $d=2$, $\O = (-1,1)^2$, $r=1/2$, and $\a=10$. 
\end{example}

\subsubsection*{Iterative solution}
For the practical solution of the minimization problem we use a regularization
defined with the regularized euclidean length 
\[
|s|_\veps = (|s|^2+\veps^2)^{1/2}
\]
for $\veps>0$ and $s\in \R^d$. The uniform approximation property
$0 \le |s|_\veps - |s| \le \veps$ for all $s\in \R^d$
implies that with the regularized functional 
\[
I_\veps(u) = \int_\O |\nabla u|_\veps \dv{x} + \frac{\a}{2}\|u-g\|^2,
\]
we have for minimizers $u$ of $I$ and $u_\veps$ of $I_\veps$ that 
\[
\frac{\a}{2} \|u-u_\veps \|^2 \le I(u_\veps) - I(u) \le \veps.
\]
This justifies using the regularized functional with $\veps=h$ to 
compute approximations for minimizers of $I$. We use Algorithm~\ref{alg:primal}
to decrease the energy and stop the iteration when 
$\|d_t u^k\| \le \veps_{\rm stop} = h/20$. We always use the $L^2$ inner
product and the step size $\tau=1$.

\subsubsection*{Experimental results}
For triangulations $\cT_\ell$ of $\O=(-1,1)^2$ resulting from $\ell\ge 0$
uniform refinements of a coarse triangulation of $\O$ into two triangles
we have that the maximal mesh-size of $\cT_\ell$ is proportional to 
$h_\ell =2^{-\ell}$. For a simple implementation we use the 
function $\tg_h\in \cL^0(\cT_h)$ via
\[
\tg_h(x_T) = g(x_T) = 
\begin{cases}
 1 & |x_T| < 1/2, \\ 0 & |x_T| \ge 1/2,
\end{cases}
\]
instead of the $L^2$ projection $g_h = \Pi_{h,0} g$. Since for $g=\chi_{B_r(0)}$
we have $\|g-\tg_h\|_{L^1(\O)} \le c h |\p B_r(0)|$, the error estimate 
remains valid.  The top row in Figure~\ref{fig:comp_p1_cr_tv} shows the 
numerical solutions obtained for the discretizations using a standard $P1$ 
method and the Crouzeix--Raviart method on the triangulation $\cT_5$. 
At first glance the $P1$ approximation appears superior as, e.g., the 
Crouzeix--Raviart approximation does not satisfy a discrete maximum
principle. The projections of the approximations onto piecewise constant
functions are shown in the bottom row of Figure~\ref{fig:comp_p1_cr_tv}
and lead to a different interpretation. The circular discontinuity
set is better resolved by the discontinuous method and we observe a more
localized approximation of the jump set.  
Figure~\ref{fig:conv_rates_cr_vs_p1_tv} supports the latter interpretation
via logarithmic plots for the experimental convergence rates of the error quantity
\[
\|e_h\|^2 = \|\Pi_{h,0} u_h - u(x_\cT) \|^2, 
\]
where $x_\cT|_T = x_T$ for every $T\in \cT_\ell$,
versus the number of vertices $N_\ell \sim h_\ell^{-2}$. We observe that the $L^2$
error for the Crouzeix--Raviart method 
converges at the quasi-optimal rate $\cO(h^{1/2})$ while the $P1$ error
is larger and decays at a lower rate. The approximations were
computed on the triangulations $\cT_\ell$ for $\ell=3,4,\dots,9$ with
$N_\ell = (2^\ell+1)^2 = 81, 289, \dots ,66049, 263169$ vertices. 

\begin{figure}[p]
\includegraphics[width=.49\linewidth]{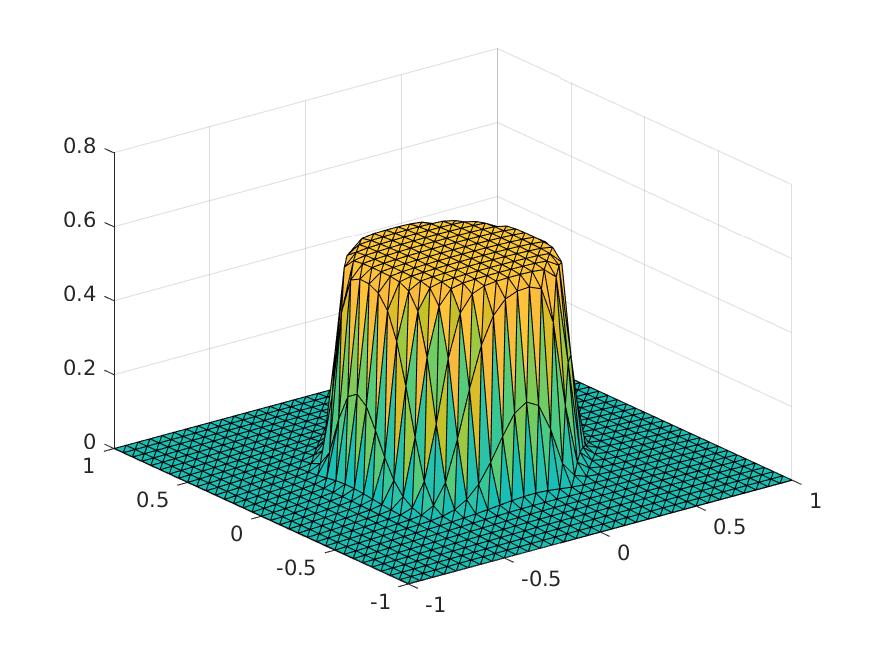} \hspace*{0.1mm}
\includegraphics[width=.49\linewidth]{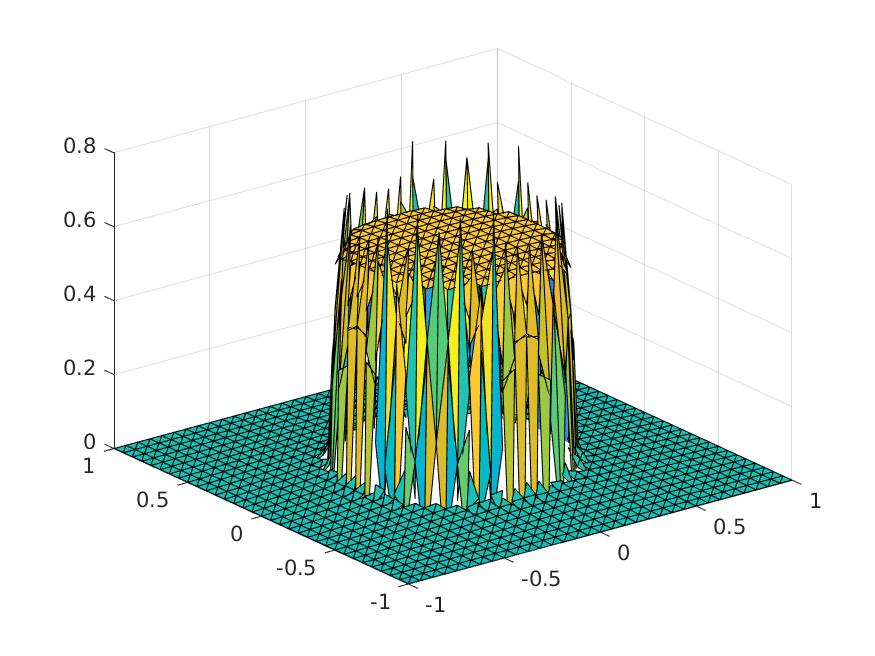} \\ 
\includegraphics[width=.49\linewidth]{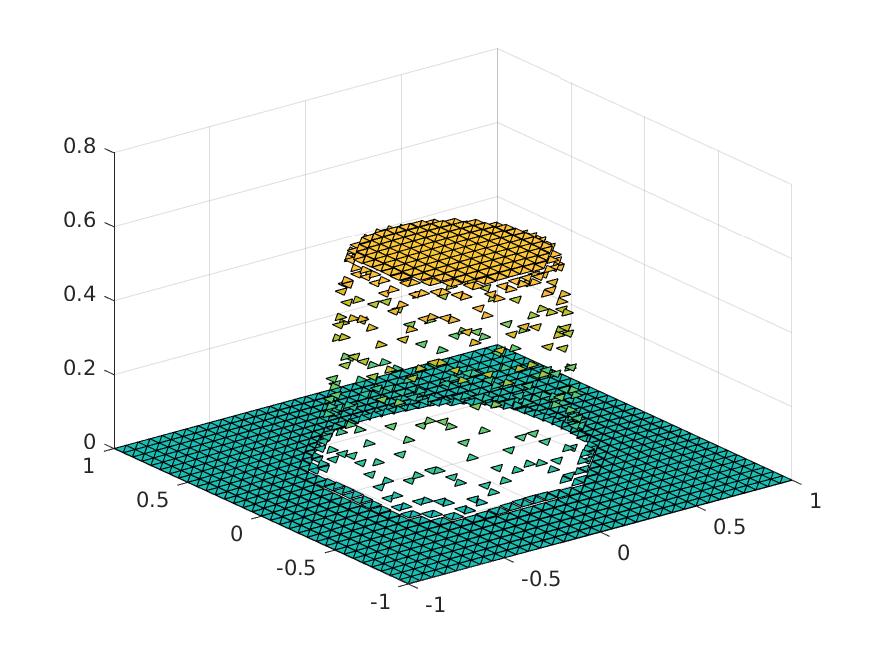} \hspace*{0.1mm}
\includegraphics[width=.49\linewidth]{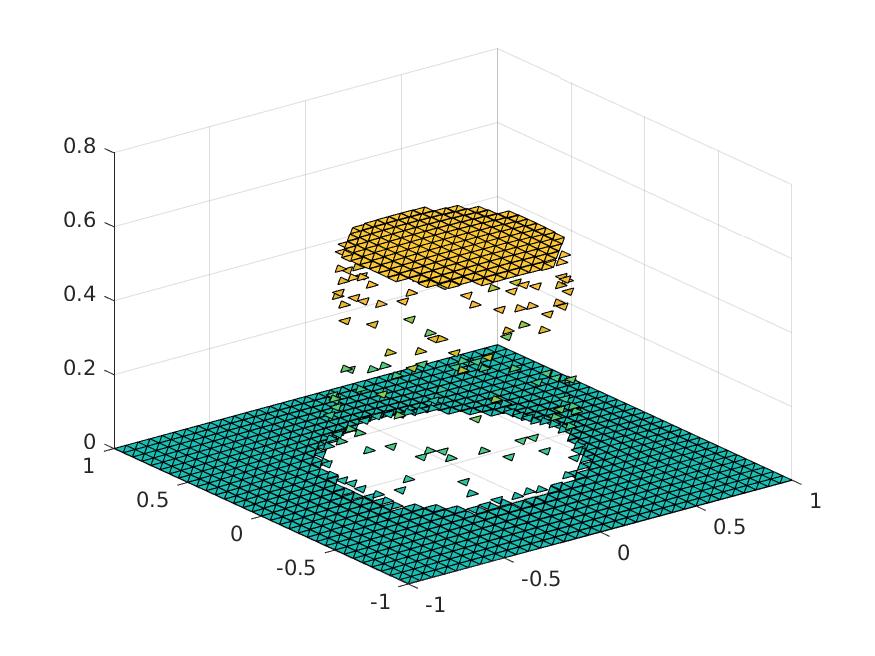} 
\caption{\label{fig:comp_p1_cr_tv} Continuous $P1$ (left) and Crouzeix--Raviart
approximations (right) in Example~\ref{ex:tv_exact} displayed 
as piecewise affine functions (top) and via their projections onto 
piecewise constant functions (bottom).}
\end{figure}

\begin{figure}[p]
\includegraphics[width=.65\linewidth]{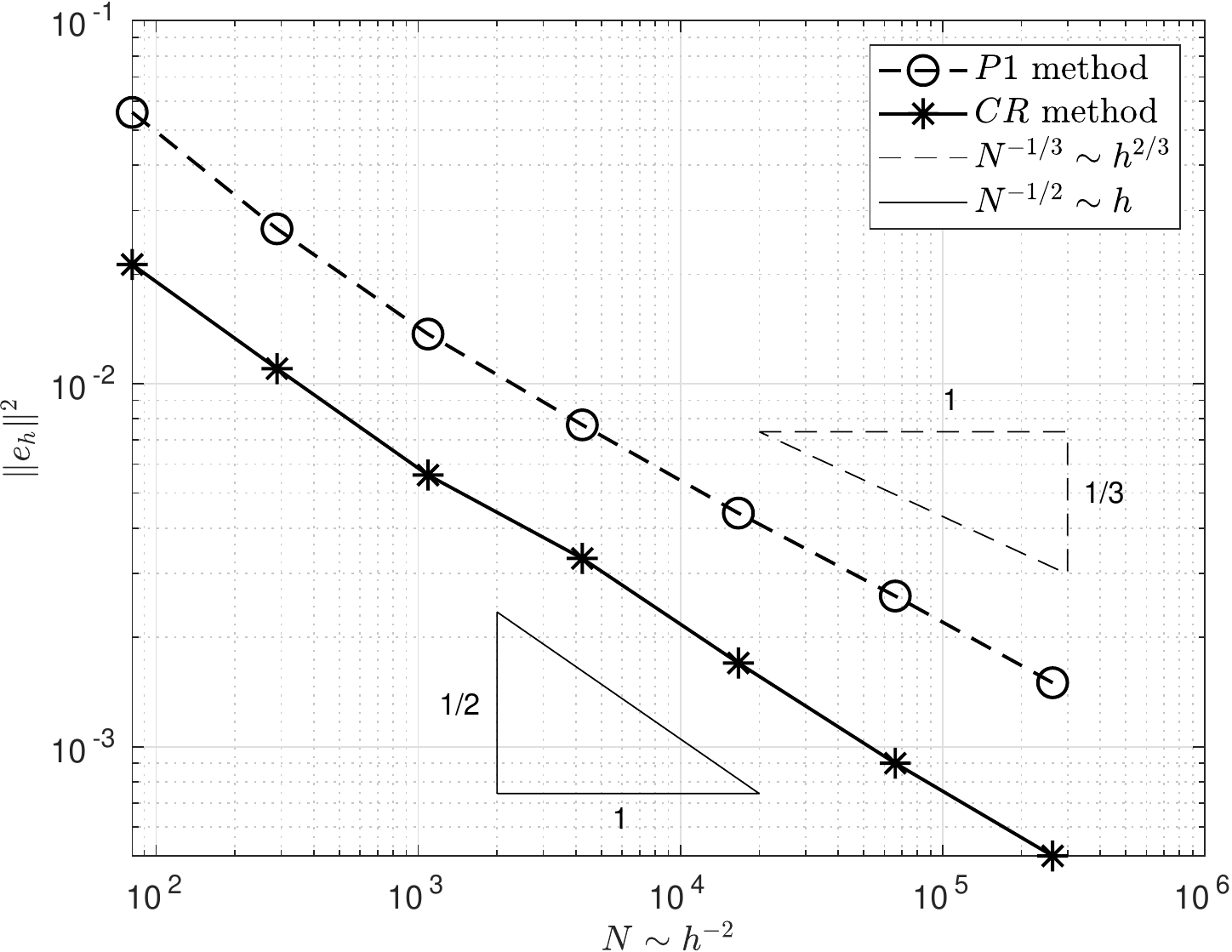}
\caption{\label{fig:conv_rates_cr_vs_p1_tv} 
Squared $L^2$ errors in Example~\ref{ex:tv_exact} for $P1$ and Crouzeix-Raviart 
approximations. The predicted rate $\cO(h^{1/2})$ is observed
for the Crouzeix--Raviart method while the 
$P1$ method leads to larger errors and a reduced rate.}
\end{figure}

\subsection{Effect of modification}
The operator $\Pi_{h,0}$ that occurs in the discrete dual problem via the term
$\phi^*(\Pi_{h,0} z_h)$  is crucial for the discrete duality theory and 
in fact simplifies the realization of the method as quadrature becomes
trivial. This does not affect the discrete flux variable $z_h$ but leads to
a modified discrete Langrange multiplier $\ou_h$. To illustrated this 
effect we consider the standard dual mixed formulation~\eqref{eq:poisson_mixed_discr} 
of the Poisson problem and the modified version which seeks 
$(z_h,\ou_h)\in \RT^0_{\!N}(\cT_h)\times \cL^0(\cT_h)$ satisfying 
\[
(\Pi_{h,0} z_h, y_h) + (\ou_h,\diver z_h) = 0, \quad (\ov_h,\diver y_h) = -(f_h,v_h),
\]
for all $(y_h,\ov_h)\in \RT^0_{\!N}(\cT_h)\times \cL^0(\cT_h)$. We specify 
the problem as follows.  

\begin{example}[Poisson problem]\label{ex:poisson}
Let $d=2$, $\O= (-1,1)^2$, $\GD = \p\O$, and $f(x,y) = 2(1-x^2)+2(1-y^2)$.
The solution of the dual mixed formulation of the Poisson problem is 
given by $u(x,y) = (1-x^2)(1-y^2)$ and $z= \nabla u$. 
\end{example} 

Figure~\ref{fig:conv_rates_rt_lump_vs_class} shows the $L^2$ errors
\[
\|e_h\| = \|\ou_h - u(x_\cT) \|, 
\]
versus numbers of vertices in $\cT_\ell$ with a logarithmic scaling
on both axes. The $L^2$ error for the modified treatment is larger
than that for the exact treatment but converges at the same quadratic
rate. This rate is higher than the expected linear convergence rate for
the difference $\|\ou_h - u\|$. An explanation is provided by the
relation $\ou_h =  u_h(x_T)$ to solutions $u_h$ of the 
Crouzeix--Raviart discretization for which we have $\|u_h-u\|_{L^\infty(\O)} 
= \cO(h^2\log(h))$, cf.~\cite{GasNoc87}.

\begin{figure}[htb]
\includegraphics[width=.65\linewidth]{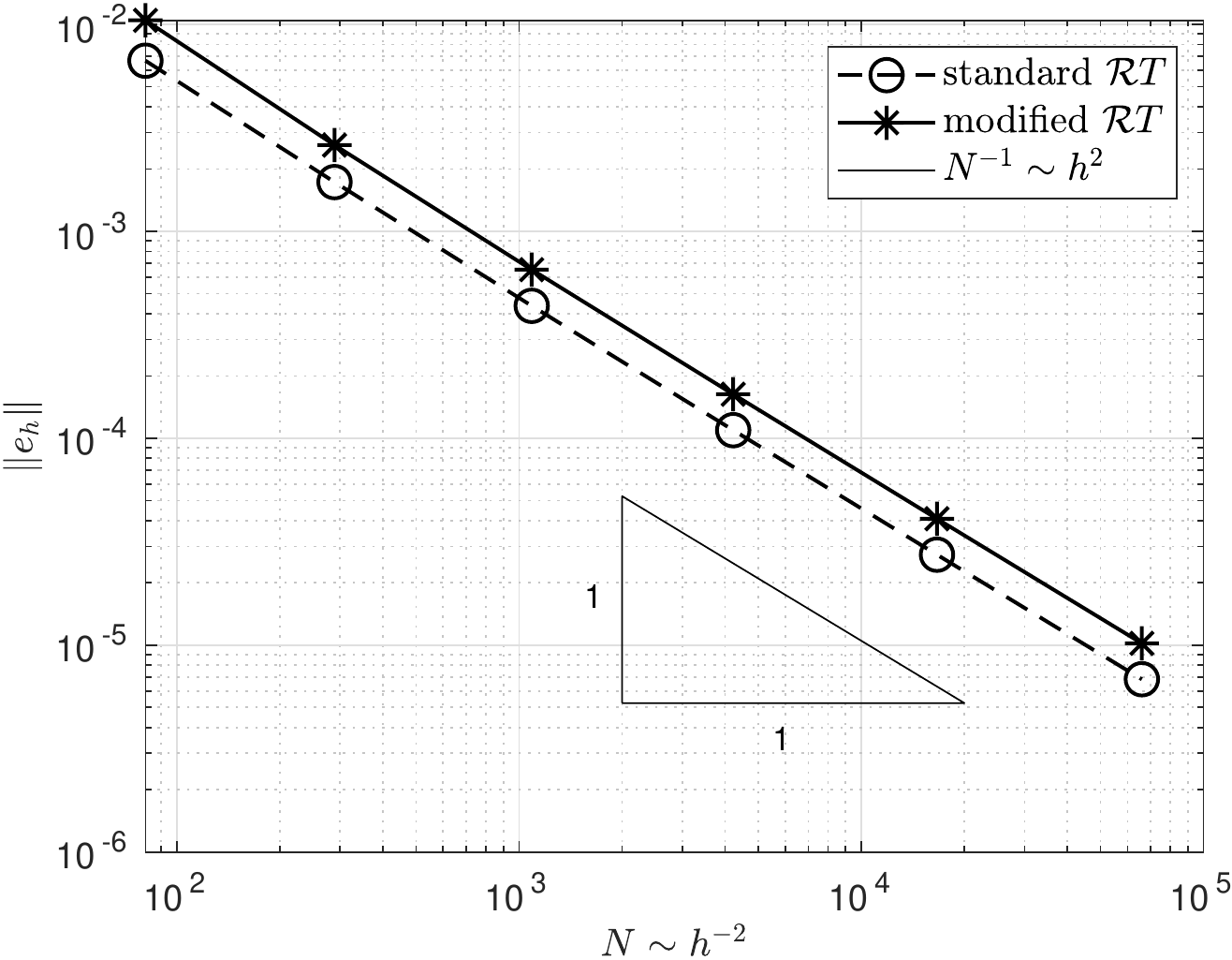}
\caption{\label{fig:conv_rates_rt_lump_vs_class} $L^2$ errors for the 
standard and modified Raviart--Thomas approximations of the Poisson problem
of Example~\ref{ex:poisson}.
An increased $L^2$ error is observed for the modified
treatment but both approximations converge with nearly qudratic rate.}
\end{figure}

\subsection{Infinity Laplacian}
We define an infinity Laplace problem via the primal functional
\[
I(u) = -\int_\O f u \dv{x}, \quad |\nabla u|\le 1,
\] 
on the set $W^{1,\infty}_D(\O)$ for a given function $f\in L^1(\O)$.
We approximate solutions by determining nearly maximizing discrete 
vector fields for the regularized dual functional
\[
D_\veps(z) = -\int_\O |z|_\veps\dv{x}, \quad -\diver z= f,
\]
with $|s|_\veps = (|s|^2 +\veps^2)^{1/2}$. 
We consider the following specification that leads to a Lipschitz
continuous solution.

\begin{example}[Infinity Laplacian]\label{ex:inf_laplace}
Let $d=2$, $\O =(-1,1)^2$, $\GD = \p\O$, and $f(x,y) =1$. Then
the solution of the primal problem is given by 
$u(x,y) = 1-\max\{|x|,|y|\}$. 
\end{example}

We use Algorithm~\ref{alg:dual} with the $L^2$ scalar product and $\tau =1$
to iteratively determine discrete
minimizers for $D_\veps$ using $\veps =h$. We also compute conforming 
approximations $u_h^c$ for the primal problem using a conforming $P1$ 
finite element method
and the ADMM iteration described in Remarks~\ref{rem:admm_primal_dual}.
Figure~\ref{fig:inf_laplace_ener} displays the resulting
approximation errors
\[
|D(z)-D_{\veps,h}(z_h)|, \quad |I(u)-I(u_h^c)|,
\]
obtained using the Raviart-Thomas method for the dual problem and
a standard conforming $P1$ method for the primal problem. We observe
that on right-angled triangulations the $P1$ method leads to an almost quadratic
convergence rate which is slightly better than the experimental convergence
rate $\cO(h^{5/3})$ observed for the Raviart--Thomas method. Surprisingly,
the nearly quadratic convergence behavior is also observed for 
$P1$ finite element approximations on perturbed triangulations. We note
however that in this case the admissibility of the nodal interpolant
is not true in general, cf.~Remark~\ref{rem:inf_laplace_p1}.

\begin{figure}[h!]
\includegraphics[width=.49\linewidth]{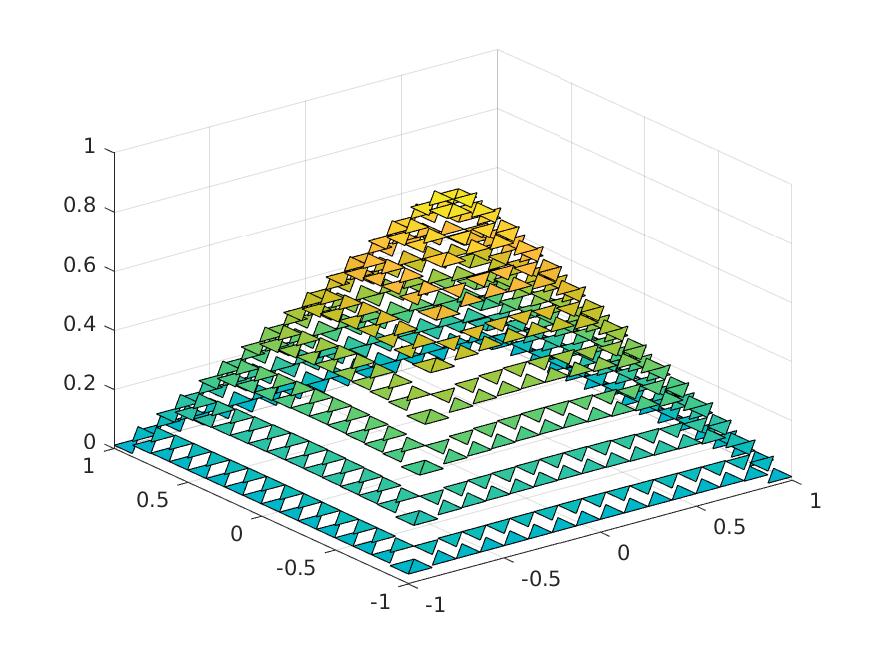} \hspace*{0.1mm}
\includegraphics[width=.49\linewidth]{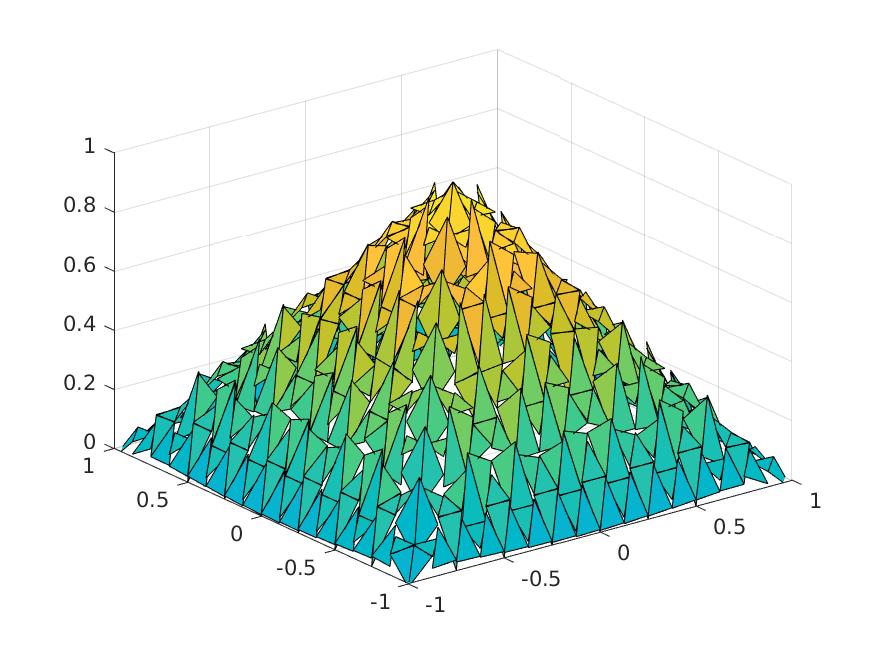} \\ 
\caption{\label{fig:sols_inf_laplace} Piecewise constant approximation of the
solution of the infinity Laplacian defined in Example~\ref{ex:inf_laplace}
obtained with the Raviart--Thomas discretization
of the regularized dual problem (left) and the corresponding reconstructed Crouzeix--Raviart 
approximation (right).}
\end{figure}

\begin{figure}[h!]
\includegraphics[width=.65\linewidth]{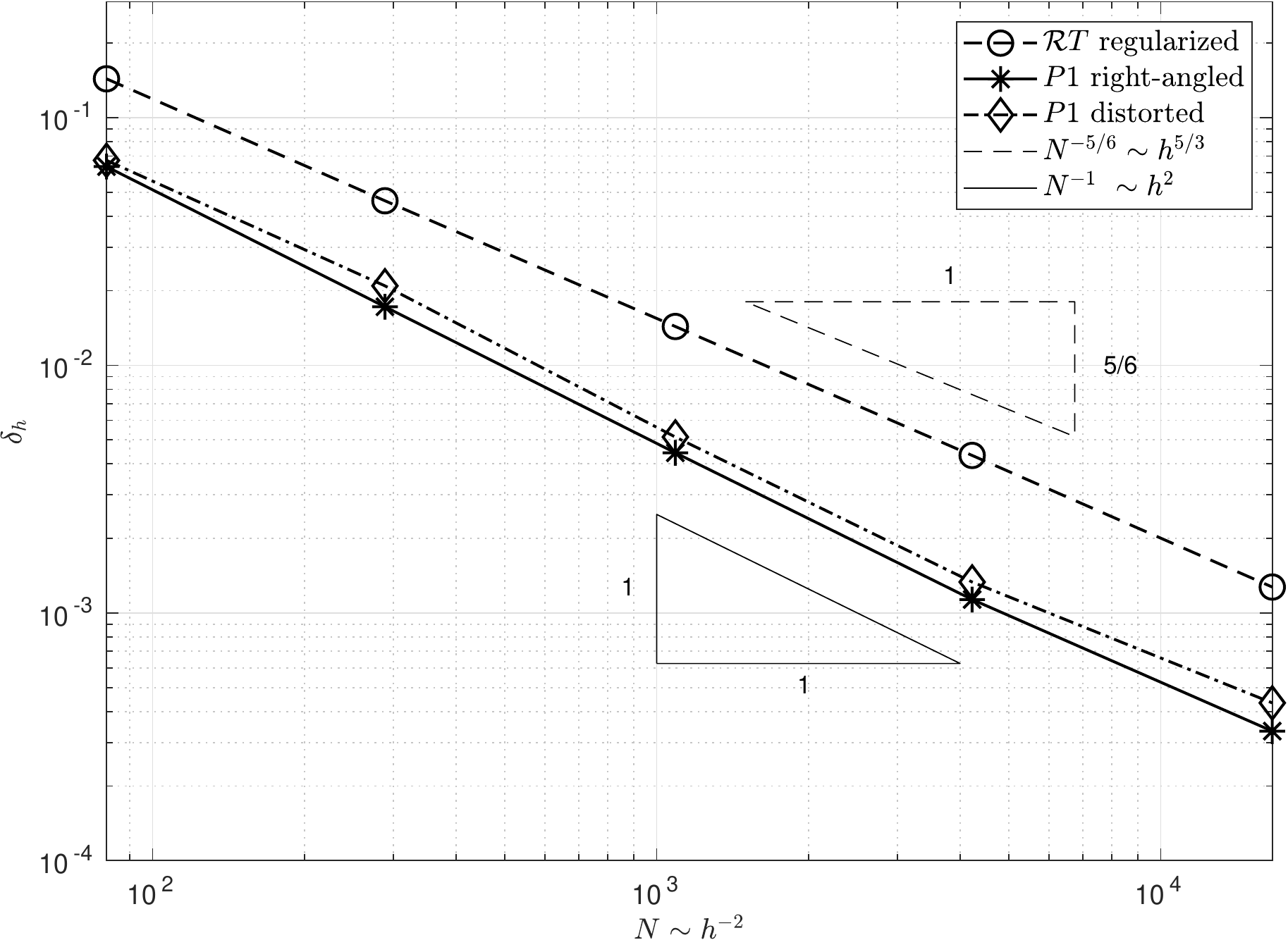}
\caption{\label{fig:inf_laplace_ener} Experimental convergence
rate for the approximation of the value $D(z)$ using the Raviart--Thomas
discretization $D_{h,\veps}$ and of $I(u)$ using a conforming $P1$ method
in the case of the inifinity Laplace problem of Example~\ref{ex:inf_laplace}.}
\end{figure}

\clearpage
\appendix
\section{Auxiliary results}

\subsection{Enriching operator}\label{app:enrich}
To prove the estimates for the enriching operator $E_h^\CR$ we follow the ideas 
described in~\cite{Bren15} and define for given $v_h\in \cS^{1,\CR}_D(\cT_h)$ 
a continuous, piecewise quadratic function 
$\a_h = E_h^\CR v_h \in \cS^{2,0}_D(\cT_h) \subset W^{1,p}_D(\O)$ 
by defining the nodal values $\a_z$ associated with vertices 
$z\in \cN_h$ and $\a_S$ associated with sides $S\in \cS_h$ via 
\[
\a_S = v_h(x_S)
\]
for all $S\in \cS_h$, and in case of vertices $\a_z = 0$ if $z\in \cN_h\cap \GD$
and otherwise
\[
\a_z = n_z^{-1} \sum_{T\in \cT_h: z \in T}  v_h|_T(z),
\]
where $n_z$ is the number of elements in $\cT_h$ that contain $z$, so that
constant functions are reproduced. With
a nodal basis $(\tvphi_\tz)_{\tz\in \cN_h \cup \cS_h}$ for $\cS^{2,0}(\cT_h)$ 
which defines a partition of unity with supports $\o_\tz$ of diameters
$h_\tz$ we have
\[\begin{split}
\|\nabla \a_h\|_{L^p(\O)}^p 
&\le \sum_{\tz \in \cN_h \cup \cS_h} \int_{\o_\tz} |\a_\tz - v_h| |\nabla \tvphi_\tz| |\nabla \a_h|^{p-1} \dv{x} \\
&\le  c \sum_{\tz \in \cN_h \cup \cS_h} \|\a_\tz - v_h \|_{L^p(\o_\tz)} h_\tz^{-1} \|\nabla \a_h \|_{L^p(\o_z)}^{p-1}.
\end{split}\]
To show that $\|\nabla \a_h \|_{L^p(\O)} \le c \|\nablah v_h\|_{L^p(\O)}$ it suffices
to prove that $\|\a_\tz - v_h \|_{L^p(\o_\tz)} \le c h_\tz \|\nablah v_h \|_{L^p(\o_\tz)}$. 
If $\tz= S$ then the piecewise affine function $w_h = \a_\tz -v_h$ vanishes at
$x_S$ and the estimate follows. If $\tz = z \in \GD$ then there exist $S\in \cS_h \cap \GD$ with
$z \in S$ and we have $\a_\tz = 0$ and $v_h(x_S) =0$ and again the estimate follows.
Finally, for $\tz = z\in\cN_h \setminus \GD$ we choose a side $S\in \cS_h$ with $z\in S$ and replace
$v_h$ by $v_h - v_h(x_S)$ which corresponds to replacing $\a_z$ by $\a_z - v_h(x_S)$.
In particular, the difference $\a_z - v_h$ remains unchanged.  Hence, we assume $v_h(x_S)= 0$ and estimate,
using $\|v_h\|_{L^\infty(T)} \le c h_z^{-d/p} \|v_h\|_{L^p(T)}$,   
\[
\|\a_z\|_{L^p(\o_z)} \le |\a_z| |\o_z|^p \le c h_z^{d/p} n_z^{-1} \sum_{T\in \cT_h,\, z\in T} \big|v_h|_T(z)\big| 
\le c \|v_h\|_{L^p(\o_z)}.
\]
Using that $v_h(x_S) = 0$ this implies that
\[
\|\a_z - v_h \|_{L^p(\o_z)}  \le c \|v_h\|_{L^p(\o_z)} \le c h_z \|\nablah v_h\|_{L^p(\o_z)}.
\]
The estimate $\|E_h^\CR v_h -v_h\|_{L^p(\O)}\le c h \|\nablah v_h\|_{L^p(\O)}$ follows from
noting that for every $\tz \in \cN_h\cup \cS_h$ we may choose $S\in \cS_h$ belonging to $\o_\tz$
and hence 
\[\begin{split}
\|E_h^\CR v_h -v_h\|_{L^p(\o_\tz)} 
&\le \|E_h^\CR v_h -v_h(x_S)\|_{L^p(\o_\tz)} + \|v_h -v_h(x_S)\|_{L^p(\o_\tz)} \\
&\le c h_\tz \big( \|\nabla E_h^\CR v_h \|_{L^p(\o_\tz)} + \|\nablah v_h\|_{L^p(\o_\tz)}\big),
\end{split}\]
since $E_h^\CR v_h (x_S) = v_h(x_S)$. 

\subsection{Proof of inequality~\eqref{eq:mon_phi}}\label{app:monotone}
We assume that $\vphi \in C^1(\R_{\ge 0})$ is convex 
and that $r\mapsto \vphi'(r)/r$ is positive, nonincreasing, 
and continuous on $\R_{\ge 0}$ and follow~\cite{BaDiNo18}. For $a,b\in \R^d$ the
identity $2 b\cdot (b-a) =  |b|^2 -|a|^2 + |b-a|^2$ yields that
\[
\frac{\vphi'(|a|)}{|a|} b\cdot (b-a) 
= \frac12 \frac{\vphi'(|a|)}{|a|} \big(|b|^2-|a|^2\big) 
 + \frac12 \frac{\vphi'(|a|)}{|a|} |b-a|^2.
\]
Since $r\mapsto \vphi'(r)/r$ is nonincreasing, 
the function $\tvphi(y) = \vphi(y^{1/2})$ is concave on $\R_{\ge 0}$,
so that we have 
\[
\tvphi'(y) (z-y) \ge \tvphi(z)-\tvphi(y),
\]
for all $y,z\ge 0$. With $y=|a|^2$ and $z=|b|^2$ we deduce that 
\[
\frac12 \frac{\vphi'(|a|)}{|a|} \big(|b|^2-|a|^2\big) 
\ge \vphi(|b|) -\vphi(|a|).
\]
Combining these inequalities implies the asserted inequality
\[
\frac{\vphi'(|a|)}{|a|} b\cdot (b-a) \ge \vphi(|b|) - \vphi(|a|) 
+ \frac12 \frac{\vphi'(|a|)}{|a|} |b-a|^2.
\]

\subsection{$p$-Dirichlet energies}\label{app:p_laplace}
For the $p$-Dirichlet energy defined via $\phi(a) = |a|^p/p$ it is shown 
in~\cite{DieKre08} that if $u\in W^{1,p}_D(\O)$ is minimal then we have with 
$F(a) = |a|^{p/2-1} a$ 
\[
c_p \|F(\nabla u)-F(\nabla v) \|^2 \le I(v)-I(u).
\]
The estimate carries over to the discretized functional $I_h$ using the Crouzeix--Raviart 
method. It is shown in~\cite{DiEbRu07} via Taylor approximations
that with $S(a) = D\phi(a) = |a|^{p-2} a$ 
and $\vphi_{|a|}(|c|) = (|a|+|c|)^{p-2} |c|^2$ we have
\[
(S(a)-S(b)) \cdot (a-b) \approx |F(a)-F(b)|^2 \approx \vphi_{|a|}(|a-b|).
\]
The relations hold also for the functionals $\tS$ and $\tF$ which are obtained
by replacing $p$ by $p'=p/(p-1)$. The article~\cite{DieRuz07} implies the estimate
\[
\|F(\nabla u) - F(\nablah \cI_\CR u) \| \le c h \|\nabla F(\nabla u)\|,
\]
provided that $F(\nabla u)\in W^{1,2}(\O;\R^d)$. 

\subsection*{Acknowledgments}
(i) Some time before his death John W. Barrett told me that he was working on
a generalization of Marini's formula, which unfortunately he was not able
to complete. His comment was an important inspiration for the results of
this article. I believe that many of the ideas presented here were known to
him. \\ 
(ii) The author acknowledges support by the DFG via the priority programmes
SPP 1748 {\em Reliable Simulation Techniques in Solid Mechanics: Development of Non-standard 
Discretization Methods, Mechanical and Mathematical Analysis} and 
SPP 1962 {\em Non-smooth and Complementarity-based Distributed Parameter Systems: 
Simulation and Hierarchical Optimization}.


\bibliographystyle{abbrv}
\bibliography{bib_tv_rt}

\end{document}